\documentclass{article}

\usepackage{amsmath,amsfonts,amssymb,amsthm,graphicx,url,subfigure,float}
\usepackage[colorlinks=true]{hyperref}
\usepackage{pdfsync}
\usepackage{color}

\def\R{\textrm{I\kern-0.21emR}}
\def\N{\textrm{I\kern-0.21emN}}

\renewcommand{\leq}{\leqslant}
\renewcommand{\geq}{\geqslant}
\renewcommand{\ker}{\mathrm{Null}}
\newcommand{\Img}{\mathrm{Range}}

\newtheorem{proposition}{Proposition}
\newtheorem{problem}{Problem}
\newtheorem{definition}{Definition}
\newtheorem{lemma}{Lemma}
\newtheorem{theorem}{Theorem}
\newtheorem{corollary}{Corollary}
\theoremstyle{definition}\newtheorem{remark}{Remark}
\theoremstyle{definition}\newtheorem{example}{Example}

\title{Shape deformation analysis from the optimal control viewpoint}

\date{}

\author{Sylvain Arguill\`ere\footnote{Sorbonne Universit\'es, UPMC Univ Paris 06, CNRS UMR 7598, Laboratoire Jacques-Louis Lions, F-75005, Paris, France ({\tt sylvain.arguillere@upmc.fr}).}
\and
Emmanuel Tr\'elat\footnote{Sorbonne Universit\'es, UPMC Univ Paris 06, CNRS UMR 7598, Laboratoire Jacques-Louis Lions, Institut Universitaire de France and Team GECO Inria Saclay, F-75005, Paris, France (\texttt{emmanuel.trelat@upmc.fr}).}
\and
Alain Trouv\'e\footnote{Ecole Normale Sup\'erieure de Cachan, Centre de Math\'ematiques et Leurs Applications, CMLA, 61 av. du Pdt Wilson, F-94235 Cachan Cedex, France ({\tt trouve@cmla.ens-cachan.fr}).}
\and
Laurent Younes\footnote{Johns Hopkins University, Center for Imaging Science, Department of Applied Mathematics and Statistics, Clark 324C, 
3400 N. Charles st. Baltimore, MD 21218, USA ({\tt laurent.younes@jhu.edu}).}
}

\begin{document}

\maketitle

\begin{abstract}
A crucial problem in shape deformation analysis is to determine a deformation of a given shape into another one, which is optimal for a certain cost. It has a number of applications in particular in medical imaging.

In this article we provide a new general approach to shape deformation analysis, within the framework of optimal control theory, in which a deformation is represented as the flow of diffeomorphisms generated by time-dependent vector fields. Using reproducing kernel Hilbert spaces of vector fields,
the general shape deformation analysis problem is specified as an infinite-dimensional optimal control problem with state and control constraints. In this problem, the states are diffeomorphisms and the controls are vector fields, both of them being subject to some constraints. The functional to be minimized is the sum of a first term defined as geometric norm of the control (kinetic energy of the deformation) and of a data attachment term providing a geometric distance to the target shape.

This point of view has several advantages. First, it allows one to model general constrained shape analysis problems, which opens new issues in this field. Second, using an extension of the Pontryagin maximum principle, one can characterize the optimal solutions of the shape deformation problem in a very general way as the solutions of constrained geodesic equations.
Finally, recasting general algorithms of optimal control into shape analysis yields new efficient numerical methods in shape deformation analysis.
Overall, the optimal control point of view unifies and generalizes different theoretical and numerical approaches to shape deformation problems, and also allows us to design new approaches.

The optimal control problems that result from this construction are infinite dimensional and involve some constraints, and thus are nonstandard. In this article we also provide a rigorous and complete analysis of the infinite-dimensional shape space problem with constraints and of its finite-dimensional approximations.
\end{abstract}


\medskip

\noindent\textbf{AMS classification:}
% 58Exx % Variational problems in infinite-dimensional spaces
58E99 % None of the above, but in MSC2010 section 58Exx
49Q10 % Optimization of shapes other than minimal surfaces
46E22 % Hilbert spaces with reproducing kernels
% 47B32 % Operators in reproducing-kernel Hilbert spaces
49J15 % Optimal control problems involving ordinary differential equations
% 32G10 % Deformations of submanifolds and subspaces
62H35 % Image analysis
%53C17 % Sub-Riemannian geometry
53C22 % Geodesics
% 58Dxx % Spaces and manifolds of mappings
58D05 % Groups of diffeomorphisms and homeomorphisms as manifolds

\newpage
\tableofcontents
\newpage

\section{Introduction}

The mathematical analysis of shapes has become a subject of growing interest in the past few decades, and has motivated the development of efficient image acquisition and segmentation methods, with applications to many domains, including computational anatomy and object recognition.

The general purpose of shape analysis is to compare two (or more) shapes in a way that takes into account their geometric properties. Two shapes can be very similar from a human's point of view, like a circle and an ellipse, but very different from a computer's automated perspective. In
\textit{Shape Deformation Analysis}, one optimizes a deformation mapping one shape onto the other and bases the analysis on its properties. This of course implies that a cost has been assigned to every possible deformation of a shape, the design of this cost function being a crucial step in the method. This approach has been used extensively in the analysis of anatomical organs from medical images (see \cite{GM}).

In this framework, a powerful and convenient approach represents deformations as flows of diffeomorphisms generated by time-dependent vector fields \cite{DGM,T1,T2}. Indeed, when considering the studied shapes as embedded in a real vector space $\R^d$, deformations of the whole space, like diffeomorphisms, induce deformations of the shape itself. 
The set of all possible deformations is then defined as the set of flows of time-dependent vector fields of a Hilbert space $V$, called space of "infinitesimal transformations", which is a subset of the space of all smooth bounded vector fields on $\R^d$.

This point of view has several interesting features, not the least of which being that the space of possible deformations is a well-defined subgroup of the group of diffeomorphisms, equipped with a structure similar to the one of a right-invariant sub-Riemannian metric \cite{SRBOOK,MBOOK}. This framework has led to the development of a family of registration algorithms called \textit{Large Deformation Diffeomorphic Metric Mapping} (LDDMM), in which the correspondence between two shapes comes from the minimization of an objective functional defined as a sum of two terms \cite{AG,BMTY,JM,MTY1,MTY2}. The first term takes into account the cost of the deformation, defined as the integral of the squared norm of the time-dependent vector field from which it arises. In a way, it is the total kinetic energy of the deformation. The second term is a data attachment penalizing the difference between the deformed shape and a target.

An appropriate class of Hilbert spaces of vector fields for $V$ is the one of \textit{reproducing kernel Hilbert spaces} (in short, RKHS) \cite{RK}, because they provide very simple solutions to the spline interpolation problem when the shape is given by a set of landmarks \cite{TY2,YBOOK}, which is an important special case since it includes most practical situations after discretization.
This framework allows one to use tools from Riemannian geometry \cite{TY2}, along with classical results from the theory of Lie groups equipped with right-invariant metrics \cite{Arnold,ABOOK,HMR,MRT,YBOOK}. 
These existing approaches provide an account for some of the geometric information in the shape, like singularities for example. However, they do not consider other intrinsic properties of the studied shape, which can also depend on the nature of the object represented by the shape. For example, for landmarks representing articulations of a robotic arm, the deformation can be searched so as to preserve the distance between certain landmarks. For cardiac motions, it may be relevant to consider deformations of the shape assuming that the movement only comes from a force applied only along the fiber structure of the muscle. In other words, it may be interesting to constrain the possible deformations (by considering non-holonomic constraints) in order to better fit the model.

In order to take into account such constraints in shape deformation problems, we propose to model these problems within the framework of \textit{optimal control theory}, where the control system would model the evolution of the deformation and the control would be the time-dependent vector field (see preliminary ideas in \cite{TY2}).

The purpose of this paper is to develop the point of view of \textit{optimal control for shape deformation analysis} as comprehensively as possible. We will show the relevance of this framework, in particular because it can be used to model constrained shapes among many other applications. 

Indeed, a lot of tools have been developed in control theory for solving optimal control problems with or without constraints. The well-known \textit{Pontryagin maximum principle} (in short PMP, see \cite{PBOOK}) provides first-order conditions for optimality in the form of Hamiltonian extremal equations with a maximization condition permitting the computation of the optimal control. It has been generalized in many ways, and a large number of variants or improvements have been made over the past decades, with particular efforts in order to be able to address optimal control problems involving general state/control constraints (see the survey article \cite{HSV} and the many references therein). The analysis is, however, mainly done in finite dimension.
%A different kind of constraints has been studied in \cite{Y1}, also using results from optimal control theory. The constraints from \cite{Y1} can be made to coincide with the ones proposed in this paper (and conversely), but the point of view adopted here is very different.
Since shape analysis has a natural setting in infinite dimension (indeed, in 2D, the shape space is typically a space of smooth curves in $\R^2$), we need to derive an appropriate infinite-dimensional variant of the PMP for constrained problems. Such a variant is nontrivial and nonstandard, given that our constrained shape analysis problems generally involve an infinite number of equality constraints.

Such a PMP will allow us to derive in a rigorous geometric setting the (constrained) geodesic equations that must be satisfied by the optimal deformations.

Moreover, modeling shape deformation problems within the framework of optimal control theory can inherit from the many numerical methods in optimal control and thus lead to new algorithms in shape analysis.

\medskip

The paper is organized as follows. 

Section \ref{sec2} is devoted to modeling shape deformation problems with optimal control.
We first briefly describe, in Section \ref{sec2.1}, the framework of diffeomorphic deformations arising from the integration of time-dependent vector fields belonging to a given RKHS, and recall some properties of RKHS's of vector fields.
In Section \ref{modshsp} we introduce the action of diffeomorphisms on a shape space, and we model and define the optimal control problem on diffeomorphisms which is at the heart of the present study, where the control system stands for the evolving deformation and the minimization runs over all possible time-dependent vector fields attached to a given RKHS and satisfying some constraints.
We prove that, under weak assumptions, this problem is well posed and has at least one solution (Theorem \ref{existence}).
Since the RKHS is in general only known through its kernel, we then provide a kernel formulation of the optimal control problem and we analyze the equivalence between both problems.
In Section \ref{mlmk} we investigate in our framework two important variants of shape spaces, which are lifted shapes and multi-shapes.
Section \ref{approx} is devoted to the study of finite-dimensional approximations of the optimal control problem.
Section \ref{sec_proof_thm_existence} contains a proof of Theorem \ref{existence}.

Section \ref{secgeod} is dedicated to the derivation of the constrained geodesic equations in shape spaces, that must be satisfied by optimal deformations.
We first establish in Section \ref{PMPs} an infinite dimensional variant of the PMP which is adapted to our setting (Theorem \ref{PMP}). As an application, we derive in Section \ref{geodeqshsp} the geodesic equations in shape spaces (Theorem \ref{geodeq}), in a geometric setting, and show that they can be written as a Hamiltonian system.

In Section \ref{sec_algos}, we design some algorithms in order to solve the optimal control problem modeling the shape deformation problem.
Problems without constraints are first analyzed in Section \ref{unconsmin}, and we recover some already known algorithms used in unconstrained shape spaces, however with a more general point of view. We are thus able to extend and generalize existing methods. Problems with constraints are investigated in Section \ref{matchconst} in view of solving constrained matching problems. We analyze in particular the augmented Lagrangian algorithm, and we also design a method based on shooting.

In Section \ref{sec7} we provide numerical examples, investigating first a matching problem with constant total volume, and then a multishape matching problem.

\section{Modelling shape deformation problems with optimal control}\label{sec2} 
The following notations will be used throughout the paper. Let $d\in \N$ fixed. A vector $a\in\R^d$ can be as well viewed as a column matrix of length $d$. The Euclidean norm of $a$ is denoted by $\vert a\vert$. The inner product $a\cdot b$ between two vectors $a,b\in\R^d$ can as well be written, with matrix notations, as $a^Tb$, where $a^T$ is the transpose of $a$. In particular one has $\vert a\vert^2=a\cdot a=a^Ta$.

Let $X$ be a Banach space. The norm on $X$ is denoted by $\Vert\cdot\Vert_X$, and the inner product by $(\cdot,\cdot)_X$ whenever $X$ is a Hilbert space. The topological dual $X^*$ of $X$ is defined as the set of all linear continuous mappings $p:X\rightarrow \R$. Endowed with the usual dual norm $\Vert p\Vert_{X^*}=\sup \{ p(x) \ \vert\ x\in X,\ \Vert x\Vert_X=1\}$, it is a Banach space.
For $p\in X^*$, the natural pairing between $p$ and $w\in X$ is $p(w)=\langle p, w\rangle_{X^*,X}$, with the duality bracket.
If $X=\R^n$ then $p$ can be identified with a column vector through the equality $p(w)=p^Tw$.

Let $M$ be an open subset of $X$, and let $Y$ be another Banach space. The Fr\'echet derivative of a map $f:M\rightarrow Y$ at a point $q\in M$ is written as $df_q$. When it is applied to a vector $w$, it is denoted by $df_q. w$ or $df_q(w)$. When $Y=\R$, we may also write $\langle df_q, w\rangle_{X^*,X}$.

We denote by $W^{1,p}(0,1;M)$ (resp. $H^{1}(0,1;M)$) the usual Sobolev space of elements of $L^p(0,1;M)$, with $1\leq p\leq+\infty$  (resp., with $p=2$) having a weak derivative in $L^p(0,1;X)$. For $q_0\in M$ we denote by $W_{q_0}^{1,p}(0,1;M)$ (resp., by $H^1_{q_0}(0,1;M)$) the space of all $q\in W^{1,p}(0,1;M)$ (resp., $q\in H^1_{q_0}(0,1;M)$) such that $q(0)=q_0$.

For every $\ell\in\N$, a mapping $\varphi:M\rightarrow M$ is called a $\mathcal{C}^\ell$ diffeomorphism if it is a bijective mapping of class $\mathcal{C}^\ell$ with an inverse of class $\mathcal{C}^\ell$. The space of all such diffeomorphisms is denoted by $\text{Diff}\, ^\ell(M)$. Note that $\text{Diff}\, ^0(M)$ is the space of all homeomorphisms of $M$. 

For every mapping $f:\R^d\rightarrow X$ of class $\mathcal{C}^\ell$ with compact support, we define the usual semi-norm
$$
\Vert f\Vert_{\ell}=\sup \left\{ \bigg\Vert \frac{\partial^{\ell_1+\dots+\ell_d}f(x)}{\partial x_1^{\ell_1}\dots\partial x_d^{\ell_d}}\bigg\Vert_X \quad \big\vert\quad   x\in \R^d,\ (\ell_1,\dots,\ell_d)\in \N^d,\ \ell_1+\dots+\ell_d\leq \ell   \right\}.
$$
%\textcolor{red}{The spaces $\mathcal{C}^k(\R^d,X)$ and $\text{Diff}\, ^k(\R^d)$ are equipped with the Fréchet structure induced by this family of all such semi-norms, for $U$ compact.}

We define the Banach space $\mathcal{C}_0^\ell(\R^d,\R^d)$ (endowed with the norm $\Vert \cdot\Vert_{\ell}$) as the completion of the space of vector fields of class $\mathcal{C}^\ell$ with compact support on $\R^d$ with respect to the norm $\Vert\cdot\Vert_{\ell}$. In other words, $\mathcal{C}_0^\ell(\R^d,\R^d)$ is the space of vector fields of class $\mathcal{C}^\ell$ on $\R^d$ whose derivatives of order less than or equal to $\ell$ converge to zero at infinity.

We define $\text{Diff}\,_0 ^\ell(\R^d)$ as the set of all diffeomorphisms of class $\mathcal{C}^\ell$ that converge to identity at infinity. Clearly, $\text{Diff}\,_0 ^\ell(\R^d)$ is the set of all $\varphi\in \text{Diff}\,^\ell(\R^d)$ such that $\varphi-\mathrm{Id}_{\R^d}\in \mathcal{C}_0^\ell(\R^d,\R^d)$. It is a group for the composition law $(\varphi,\psi)\mapsto\varphi\circ\psi$.

Note that, if $\ell\geq 1$, then $\text{Diff}\,_0 ^\ell(\R^d)$ is an open subset of the affine Banach space $\mathrm{Id}_{\R^d}+\mathcal{C}_0^\ell(\R^d,\R^d)$. This allows one to develop a differential calculus on $\text{Diff}\,_0 ^\ell(\R^d)$.

\subsection{Preliminaries: deformations and RKHS of vector fields}\label{sec2.1}
Our approach to shape analysis is based on optimizing evolving deformations. A deformation is a one-parameter family of flows in $\R^d$ generated by time-dependent vector fields on $\R^d$. Let us define this concept more rigorously.

\paragraph{Diffeomorphic deformations.}
Let $\ell\in \N^*$. Let
\begin{equation*}
\begin{array}{rcl}
v: [0,1] & \longrightarrow & \mathcal{C}_0^\ell(\R^d,\R^d) \\
t & \longmapsto & v(t)
\end{array}
\end{equation*}
be a time-dependent vector field such that the real-valued function $t\mapsto \Vert v(t)\Vert_{\ell}$ is integrable. In other words, we consider an element $v$ of the space $ L^1(0,1;\mathcal{C}_0^\ell(\R^d,\R^d))$.

According to the Cauchy-Lipshitz theorem, $v$ generates a (unique) flow $\varphi: [0,1]\rightarrow \text{Diff}\,_0 ^1(\R^d)$ (see, e.g., \cite{AgrachevSachkov} or \cite[Chapter 11]{TBOOK}), that is a one-parameter family of diffeomorphisms such that
\begin{equation*}
\begin{split}
\frac{\partial\varphi}{\partial t}(t,x)&=v(t,\varphi(t,x)),\\
\varphi(0,x)&=x,
\end{split}
\end{equation*}
for almost every $t\in [0,1]$ and every $x\in \R^d$. In other words, considering $\varphi$ as a curve in the space $\text{Diff}\,_0 ^1(\R^d)$, the flow $\varphi$ is the unique solution of
\begin{equation}\label{EDOdiffeo}
\begin{split}
\dot{\varphi}(t)&=v(t)\circ\varphi(t),\\
\varphi(0)&=\mathrm{Id}_{\R^d}.
\end{split}
\end{equation}
Such a flow $\varphi$ is called a \textit{deformation} of $\R^d$ of class $\mathcal{C}^\ell$.

\begin{proposition} 
The set of deformations of $\R^d$ of class $\mathcal{C}^\ell$ coincides with the set
$$
\left\{ \varphi\in W^{1,1}(0,1;\text{Diff}\,_0^\ell(\R^d)) \mid \varphi(0)=\mathrm{Id}_{\R^d} \right\}.
$$
In other words, the deformations of $\R^d$ of class $\mathcal{C}^\ell$ are exactly the curves $t\mapsto\varphi(t)$ on $\text{Diff}\,_0^\ell(\R^d)$ that are integrable on $(0,1)$ as well as their derivative, such that $\varphi(0)=\mathrm{Id}_{\R^d}$.
\end{proposition}

\begin{proof}
Let us first prove that there exists a sequence of positive real numbers $(D_n)_{n\in\N}$ such that for every deformation $\varphi$ of $\R^d$ of class $\mathcal{C}^\ell$, with $\ell\in\N^*$, induced by the time-dependent vector field $v\in L^1(0,1;\mathcal{C}_0^\ell(\R^d,\R^d))$, one has
\begin{equation}\label{derdiffeo}
\sup_{t\in[0,1]}\Vert \varphi(t)-\mathrm{Id}_{\R^d}\Vert_{i}\leq D_i\mathrm{exp}\Big(D_i\int_0^1\Vert v(t)\Vert_{i}\Big),
\end{equation}
for every $i\in \{0,\ldots,\ell\}$.

The case $i=0$ is an immediate consequence of the integral formulation of \eqref{EDOdiffeo}. Combining the formula for computing derivatives of a composition of mappings with an induction argument shows that the derivatives of order $i$ of $v\circ\varphi$ are polynomials in the derivatives of $v$ and $\varphi$ of order less than or equal to $i$. Moreover, these polynomials are of degree one with respect to the derivatives of $v$, and also of degree one with respect to the derivatives of $\varphi$ of order $i$. Therefore we can write
\begin{equation}\label{derdiffeo2}
\left\vert \frac{d}{dt}\partial_x^i\varphi(t,x)\right\vert\leq\Vert v(t)\Vert_{i}\vert \partial_x^i\varphi(t,x)\vert+\Vert v(t)\Vert_{i-1}P_i(\vert \partial_x^1\varphi(t,x)\vert,\dots,\vert \partial_x^{i-1}\varphi(t,x)\vert),
\end{equation}
where $P_i$ is a polynomial independent of $v$ and $\varphi$, and the norms of the derivatives of the $\partial_x^{j}\varphi(t,x)$ are computed in the space of $\R^d$-valued multilinear maps. The result then follows from Gronwall estimates and from an induction argument on $i$.

That any deformation of $\R^d$ of class $\mathcal{C}^\ell$ is a curve of class $W^{1,1}$ in $\text{Diff}\,_0^\ell(\R^d)$ is then a direct consequence of \eqref{derdiffeo} and \eqref{derdiffeo2}. 

Conversely, for every curve $\varphi$ on $\text{Diff}\,_0^\ell(\R^d)$ of class $W^{1,1}$, we set $v(t)=\dot{\varphi}(t)\circ\varphi^{-1}(t),$ for every $t\in[0,1]$. We have $\dot{\varphi}(t)=v(t)\circ\varphi(t)$ for almost every $t\in[0,1]$, and hence it suffices to prove that $t\mapsto\Vert v(t)\Vert_\ell$ is integrable. The curve $\varphi$ is continuous on $[0,1]$ and therefore is bounded. This implies that $t\mapsto\varphi(t)^{-1}$ is bounded as well. The formula for computing derivatives of compositions of maps then shows that $\Vert v(t)\Vert_\ell$ is integrable whenever $t\mapsto\Vert \dot\varphi(t)\Vert_\ell$ is integrable, which completes the proof since $\varphi$ is of class $W^{1,1}$.
\end{proof}

\paragraph{Reproducing Kernel Hilbert Spaces of vector fields.}
Let us briefly recall the definition and a few properties of RKHS's (see \cite{RK,TY2} for more details). Let $k\in\N^*$ be fixed.

Given a Hilbert space $(V,(\cdot,\cdot)_V)$, according to the Riesz representation theorem, the mapping $v\mapsto (v,\cdot)_V$ is a bijective isometry between $V$ and $V^*$, whose inverse is denoted by $K_V$. Then for every $p\in V^*$ and every $v\in V$ one has $\langle p,v\rangle_{V^*,V}= (K_Vp,v)_V$ and $\Vert p\Vert_{V^*}^2=\Vert K_Vp\Vert_V^2=\langle p, K_Vp\rangle_{V^*,V}$.

\begin{definition} 
A Reproducing Kernel Vector Space (RKHS) of vector fields of class $\mathcal{C}_0^\ell$ is a Hilbert space $(V,(\cdot,\cdot)_V)$ of vector fields on $\R^d,$ such that $V\subset \mathcal{C}^\ell_0(\R^d,\R^d)$ with continuous inclusion.
\end{definition}

Let $V$ be an RKHS of vector fields of class $\mathcal{C}_0^\ell$. Then, for all $(b,y)\in(\R^d)^2$, by definition the linear form $b\otimes\delta_y$ on $V$, defined by $b\otimes\delta_y(v)=b^Tv(y)$ for every $v\in V$, is continuous (actually this continuity property holds as well for every compactly supported vector-valued distribution of order at most $\ell$ on $\R^d$).
By definition of $K_V$, there holds $\langle b\otimes\delta_y,v\rangle_{V^*,V} = ( K_V(b\otimes \delta_y) , v )_V$.
The \textit{reproducing kernel} $K$ of $V$ is then the mapping defined on $\R^d\times\R^d$, with values in the set of real square matrices of size $d$, defined by 
\begin{equation}\label{defnoyauK}
K(x,y)b=K_V(b\otimes \delta_y)(x),
\end{equation}
for all $(b,x,y)\in(\R^d)^3$. 
In other words, there holds $(K(\cdot,y)b,v)_V=b^Tv(y)$, for all $(b,y)\in(\R^d)^2$ and every $v\in V$, and $K(\cdot,y)b=K_V(b\otimes \delta_y)$ is a vector field of class $\mathcal{C}^\ell$ in $\R^d$, element of $V$.

It is easy to see that $(K(\cdot,x)a,K(\cdot,y)b)_V=a^TK(x,y)b$, for all $(a,b,x,y)\in(\R^d)^4$, and hence that $K(x,y)^T=K(y,x)$ and that $K(x,x)$ is positive semi-definite under the assumption that no nontrivial linear combination $a_1^Tv(x_1)+ \cdots +a_n^Tv(x_n)$, with given distinct $x_j$'s can vanish for every $v\in V$.
Finally, writing $K_V(a\otimes\delta_y)(x) = K(x,y)a = \int_{\R^d} K(x,s)a\, d\delta_y(s)$, we have
\begin{equation}\label{formuledistribution}
K_Vp(x)=\int_{\R^d}K(x,y) \, dp(y),
\end{equation}
for every compactly supported vector-valued distribution $p$ on $\R^d$ of order less than or equal to $\ell$.\footnote{Indeed, it suffices to note that 
$$
b^TK_Vp(x)=\left(b\otimes\delta_x,K_Vp\right)_{V^*,V}=\left(p,K_V b\otimes\delta_x\right)=\int_{\R^d}(K(y,x)b)^Tdp(y)=b^T\int_{\R^d}K(x,y)dp(y).
$$
}

As explained in \cite{RK,YBOOK}, one of the interests of such a structure is that we can define the kernel itself instead of defining the space $V$. Indeed a given kernel $K$ yields a unique associated RKHS. It is usual to consider kernels of the form $K(x,y)=\gamma(\vert x-y\vert)Id_{\R^d}$ with $\gamma\in\mathcal{C}^\infty (\R)$.
Such a kernel yields a metric that is invariant under rotation and translation. The most common model is when $\gamma$ is a Gaussian function but other families of kernels can be used as well \cite{TY1,YBOOK}.

\subsection{From shape space problems to optimal control}\label{modshsp} 
We define a \textit{shape space} in $\R^d$ as an open subset $M$ of a Banach space $X$ on which the group of diffeomorphisms of $\R^d$ acts in a certain way. The elements of $M$, called \textit{states of the shape}, are denoted by $q$. They are usually subsets or immersed submanifolds of $\R^d$, with a typical definition of the shape space as the set $M=\mathrm{Emb}^1(S,\R^d)$ of all embeddings of class $\mathcal{C}^1$ of a given Riemannian manifold $S$ into $\R^d$. For example, if $S$ is the unit circle then $M$ is the set of all parametrized $\mathcal{C}^1$ simple closed curves in $\R^d$.  
In practical applications or in numerical implementations, one has to consider finite-dimensional approximations, so that $\mathcal{S}$ usually just consists of a finite set of points, and then $M$ is a space of landmarks (see \cite{TY1,YBOOK} and see examples further).

Let us first explain how the group of diffeomorphisms acts on the shape space $M$, and then in which sense this action induces a control system on $M$.

\paragraph{The group structure of $\text{Diff}\,_0^\ell(\R^d)$.} 
Let $\ell\in \N^*$. The set $\text{Diff}\,_0^\ell(\R^d)$ is an open subspace of the affine Banach space $\mathrm{Id}_{\R^d}+\mathcal{C}_0^\ell(\R^d,\R^d)$ and also a group for the composition law. However, we can be more precise. 

First of all, the mappings $(\varphi,\psi)\mapsto \varphi\circ\psi$ and $\varphi\mapsto\varphi^{-1}$ are continuous (this follows from the formula for the computation of the derivatives of compositions of mappings).

Moreover, for every $\psi\in\text{Diff}\,_0^\ell(\R^d)$, the right-multiplication mapping $\varphi\mapsto R_\psi(\varphi)=\varphi\circ\psi$ is Lipschitz and of class $\mathcal{C}^1$, as the restriction of the continuous affine map $(\mathrm{Id}_{\R^d}+v)\mapsto (\mathrm{Id}_{\R^d}+v)\circ \psi$. Its derivative $(dR_{\psi})_{Id_{\R^d}}:\mathcal{C}_0^\ell(\R^d,\R^d)\rightarrow\mathcal{C}_0^\ell(\R^d,\R^d)$ at $\mathrm{Id}_{\R^d}$ is then given by $v\mapsto v\circ\psi$. Moreover, $(v,\psi)\mapsto v\circ\psi$ is easily seen to be continuous.

Finally,  the mapping 
\begin{equation*}
\begin{array}{rcl}
\mathcal{C}^{\ell+1}_0(\R^d,\R^d)\times\text{Diff}\,_0^\ell(\R^d)&\rightarrow&\ \mathcal{C}^\ell(\R^d,\R^d) \\
(v,\psi)&\mapsto &\ v\circ\psi
\end{array}
\end{equation*}
is of class $\mathcal{C}^{1}$. Indeed we have $\Vert v\circ(\psi+\delta\psi)-v\circ\psi-dv_{\psi}.\delta\psi\Vert_{\ell}=o(\Vert \delta\psi\Vert_{\ell})$, for every $\delta \psi\in\mathcal{C}_0^\ell(\R^d,\R^d)$.
Then, using the uniform continuity of any derivative $d^iv$ of order $i\leq\ell$, it follows that the mapping $\psi\mapsto dv_{\psi}$ is continuous. These properties are useful for the study of the Fr\'echet Lie group structure of $\text{Diff}\,^{\infty}_0(\R^d)$ \cite{OBOOK}.

\paragraph{Group action on the shape space.}
In the sequel, we fix $\ell\in\N$, and we assume that the space $\text{Diff}_0^{\max(1,\ell)}(\R^d)$ acts continuously on $M$ (recall that $M$ is an open subset of a Banach space $X$) according to a mapping
\begin{equation}\label{action}
\begin{array}{rcl}
\text{Diff}_0 ^{\max(1,\ell)}(\R^d)\times M & \rightarrow & M \\
(\varphi,q) & \mapsto & \varphi\cdot q ,
\end{array}
\end{equation}
such that $\mathrm{Id}_{\R^d}\cdot q=q$ and $\varphi\cdot(\psi\cdot q)=(\varphi\circ\psi)\cdot q$ for every $q\in M$ and all $(\varphi,\psi)\in (\text{Diff}^{\max(1,\ell)}(\R^d))^2$.

\begin{definition}\label{infgaction}
$M$ is a \textit{shape space of order} $\ell\in \N$ if the action \eqref{action} is compatible with the properties of the group structure of $\text{Diff}\,_0^{\max(1,\ell)}(\R^d)$ described above, that is:
\begin{itemize}
\item For every $q\in M$ fixed, the mapping $\varphi\mapsto\varphi\cdot q$ is Lipschitz with respect to the (weaker when $\ell=0$) norm $\Vert\cdot\Vert_{\ell}$, i.e., there exists $\gamma>0$ such that
\begin{equation}\label{lipac}
\Vert \varphi_1\cdot q-\varphi_2\cdot q\Vert_{X}\leq \gamma\Vert \varphi_1-\varphi_2\Vert_{\ell}
\end{equation}
for all $(\varphi_1,\varphi_2)\in (\text{Diff}\,_0^{\max(1,\ell)}(\R^d))^2$.
\item The mapping $\varphi\mapsto\varphi\cdot q$ is differentiable at $\mathrm{Id}_{\R^d}$. This differential is denoted by $\xi_q$ and is called the \textit{infinitesimal action} of $\mathcal{C}_0^{\max(1,\ell)}(\R^d,\R^d)$. From \eqref{lipac} one has
$$
\Vert\xi_qv\Vert_{X}\leq \gamma\Vert v\Vert_{\ell},
$$
for every $v\in\mathcal{C}_0^{\ell}(\R^d,\R^d)$, and if $\ell=0$ then $\xi_q$ has a unique continuous extension to the whole space $\mathcal{C}_0^0(\R^d,\R^d)$.
\item The mapping
\begin{equation}\label{def_xi}
\begin{array}{rcl}
\xi : M\times \mathcal{C}^{\ell}_0(\R^d,\R^d) \longrightarrow X \\
(q,v)\longmapsto\xi_qv
\end{array}
\end{equation}
is continuous, and its restriction to $M\times \mathcal{C}^{\ell+1}_0(\R^d,\R^d) $ is of class $\mathcal{C}^{1}$. In particular the mapping $q\mapsto\xi_qv$ is of class $\mathcal{C}^{1}$, for every bounded vector field $v$ of class $\mathcal{C}^{\ell+1}$.
\end{itemize}
\end{definition}

\begin{example}
For $\ell\geq 1$, the action of $\text{Diff}\,_0^\ell(\R^d)$ on itself by left composition makes it a shape space of order $\ell$ in $\R^d$.
\end{example}

\begin{example}\label{example1}
Let $\ell\in \N$ and let $S$ be a $\mathcal{C}^\ell$ smooth compact Riemannian manifold. Consider the space $M=X=\mathcal{C}^\ell(S,\R^d)$ equipped with its usual Banach norm. Then $M$ is a shape space of order $\ell$, where the action of $\text{Diff}\,_0^{\max(1,\ell)}(\R^d)$ is given by the composition $\varphi\cdot q=\varphi\circ q$. Indeed, it is continuous thanks to the rule for computing derivatives of a composition, and we also have
$$
\Vert \varphi_1\cdot q-\varphi_2\cdot q \Vert_{X}\leq\gamma\Vert \varphi_1-\varphi_2\Vert_{\ell}.
$$
Moreover, given $q\in M$ and $v\in\mathcal{C}^\ell_0(\R^d,\R^d)$, $\xi_qv$ is the vector field along $q$ given by $\xi_q(v)=v\circ q\in \mathcal{C}^\ell(M,\R^d)$. Finally, the formula for computing derivatives of a composition yields
$$
\Vert v\circ(q+\delta q)-v\circ q-dv_{q}.\delta q\Vert_{X}=o(\Vert \delta q\Vert_{X}),
$$
for every $\delta q\in M$, and the last part of the definition follows. This framework describes most of shape spaces.

An interesting particular case of this general example is when $S=(s_1,\dots,s_n)$ is a finite set (zero-dimensional manifold), $X=(\R^d)^n$ and 
$$
M=\mathrm{Lmk}_d(n)= \lbrace (x_1,\dots,x_n)\in(\R^d)^n\ \vert\ x_i\neq x_j\ \text{if}\ i\neq j\rbrace 
$$ 
is a (so-called) space of $n$ landmarks in $\R^d$. For $q=(x_1,\dots,x_n)$, the smooth action of order $0$ is $\varphi\cdot q=(\varphi(x_1),\dots,\varphi(x_n))$.
For $v\in\mathcal{C}^0_0(\R^2,\R^2)$, the infinitesimal action of $v$ at $q$ is given by $\xi_q(v)=(v(x_1),\dots,v(x_n))$.
\end{example}

\begin{remark}
In most cases, and in all examples given throughout this paper, the mapping $\xi$ restricted to $M\times \mathcal{C}^{\ell+k}_0(\R^d,\R^d)$ is of class $\mathcal{C}^k$, for every $k\in \N$.
\end{remark}

\begin{proposition}
For every $q\in M$, the mapping $\varphi\mapsto\varphi\cdot q$ is of class $\mathcal{C}^1$, and its differential at $\varphi$ is given by $\xi_{\varphi\cdot q}dR_{\varphi^{-1}}$.
In particular, given $q_0\in M$ and given $\varphi$ a deformation of $\R^d$ of class $\mathcal{C}^{\max{1,\ell}}_0$, which is the flow of the time-dependent vector field $v$, the curve $t\mapsto q(t)=\varphi(t)\cdot q_0$ is of class $W^{1,1}$ and one has 
\begin{equation}\label{edoq}
\dot{q}(t)=\xi_{\varphi(t)\cdot q_0}\dot{\varphi}(t)\circ\varphi(t)^{-1}=\xi_{q(t)}v(t),
\end{equation}
for almost every $t\in[0,1]$.
\end{proposition}

\begin{proof}
Let $q_0\in M$, fix $\varphi\in \text{Diff}\,^\ell_0(\R^d)$ and take $\delta \varphi\in T_{\varphi}\text{Diff}\,_0^\ell(\R^d)=\mathcal{C}^\ell_0(\R^d,\R^d)$. Then $\varphi+\delta\varphi\in\text{Diff}\,_0^\ell(\R^d)$ for $\Vert\delta \varphi\Vert_\ell$ small enough. We define $v=(dR_{\varphi^{-1}})_{\varphi}\delta \varphi=\delta\varphi\circ\varphi^{-1}$. We have
$$
(\varphi+\delta\varphi)\cdot q=(\mathrm{Id}_{\R^d}+v)\cdot(\varphi\cdot q)=\varphi\cdot q +\xi_{\varphi\cdot q}v+o(v)=\varphi\cdot q +\xi_{\varphi\cdot q}\dot{\varphi}\circ\varphi^{-1}+o(\delta\varphi),
$$
and therefore the mapping $\varphi\mapsto\varphi\cdot q$ is differentiable at $\varphi$, with continuous differential $\xi_{\varphi\cdot q}dR_{\varphi^{-1}}$. The result follows.
\end{proof}

The result of this proposition shows that the shape $q(t) = \varphi(t)\cdot q_0$ is evolving in time according to the differential equation \eqref{edoq}, where $v$ is the time-dependent vector field associated with the deformation $\varphi$.

At this step we make a crucial connection between shape space analysis and control theory, by adopting another point of view.
The differential equation \eqref{edoq} can be seen as a control system on $M$, where the time-dependent vector field $v$ is seen as a control.
In conclusion, the group of diffeomorphisms acts on the shape space $M$, and this action induces a control system on $M$.

As said in the introduction, in shape analysis problems, the shapes are usually assumed to evolve in time according to the minimization of some objective functional \cite{TY2}. With the control theory viewpoint developed above, this leads us to model the shape evolution as an optimal control problem settled on $M$, that we define hereafter.

\paragraph{Induced optimal control problem on the shape space.}
We assume that the action of $\text{Diff}\,^{\max(\ell,1)}_0(\R^d)$ on $M$ is smooth of order $\ell\in\N$. Let $(V,( \cdot,\cdot) _V)$ be an RKHS of vector fields of class $\mathcal{C}_0^{\ell}$ on $\R^d$. Let $K$ denote its reproducing kernel (as defined in Section \ref{sec2.1}). Let $Y$ be another Banach space. Most problems of shape analysis can be recast as follows.

\begin{problem}\label{diffprob1}
Let $q_0\in M$, and let $C:M\times V\rightarrow Y$ be a mapping such that $v\mapsto C(q,v)=C_qv$ is linear for every $q\in M$. 
Let $g:M\rightarrow\R$ be a function.
We consider the problem of minimizing the functional
\begin{equation}\label{defJ1}
J_1(q,v)=\frac{1}{2}\int_0^1\Vert v(t)\Vert^2_V \, dt+g(q(1))
\end{equation}
over all $(q(\cdot),v(\cdot))\in W_{q_0}^{1,1}(0,1;M)\times L^2(0,1;V)$ such that $\dot{q}(t)=\xi_{q(t)} v(t)$ and $C_{q(t)}v(t)=0$ for almost every $t\in [0,1]$.
\end{problem}

In the problem above, $q_0$ stands for an initial shape, and $C$ stands for continuous constraints. Recall that the infinitesimal action can be extended to the whole space $\mathcal{C}^{\ell}_0(\R^d,\R^d)$. 

Note that if $t\mapsto v(t)$ is square-integrable then $t\mapsto \dot{q}(t)$ is square-integrable as well. Indeed this follows from the differential equation $\dot{q}(t)=\xi_{q(t)} v(t)$ and from Gronwall estimates. Therefore the minimization runs over the set of all $(q(\cdot),v(\cdot))\in H^1_{q_0}(0,1;M)\times L^2(0,1;V)$.

Problem \ref{diffprob1} is an infinite-dimensional optimal control problem settled on $M$, where the state $q(t)$ is a shape and the control $v(\cdot)$ is a time-dependent vector field.
The constraints $C$ can be of different kinds, as illustrated further.
A particular but important case of constraints consists of \textit{kinetic} constraints, i.e., constraints on the speed $\dot{q}=\xi_qv$ of the state, which are of the form $C_{q(t)}\xi_{q(t)}v(t)=C_{q(t)}\dot{q}(t)=0$. Pure state constraints, of the form $C(q(t))=0$ with a differentiable map $C:M\rightarrow Y$, are in particular equivalent to the kinetic constraints $dC_{q(t)}. \xi_{q(t)}v(t)=0$.

To the best of our knowledge, except very few studies (such as \cite{Y1}), only unconstrained problems have been studied so far (i.e., with $C=0$).
In contrast, the framework that we provide here is very general and permits to model and solve far more general constrained shape deformation problems.

\begin{remark}\label{rempb1_unknowns}
Assume $V$ is an RKHS of class $\mathcal{C}^{1}_0$, and let $v(\cdot)\in L^2(0,1;V)$. Then $v$ induces a unique deformation $t\mapsto \varphi(t)$ on $\R^d$, and the curve $t\mapsto q_v(t)=\varphi(t)\cdot q_0$ satisfies $q(0)=q_0$ and $\dot{q}_v(t)=\xi_{q_v(t)} v(t)$ for almost every $t\in[0,1]$. As above, it follows from the Gronwall lemma that $q\in H^1_{q_0}(0,1;M)$. Moreover,  according to the Cauchy-Lipshitz theorem, if $\ell \geq 1$ then $q(\cdot)$ is the unique such element of $H^{1}_{q_0}(0,1;M)$. Therefore, if $\ell \geq 1$ then Problem \ref{diffprob1} is equivalent to the problem of minimizing the functional $v\mapsto J_1(v,q_v)$ over all $v\in L^2(0,1;V)$ such that $C_{q_v(t)}v(t)=0$ for almost every $t\in[0,1]$.
\end{remark}

Concerning the existence of an optimal solution of Problem \ref{diffprob1}, we need the following definition.

\begin{definition}
\label{compsup} A state $q$ of a shape space $M$ of order $\ell$ is said to have compact support if for some compact subset $U$ of $\R^d$, for some $\gamma>0$ and for all $(\varphi_1,\varphi_2)\in (\text{Diff}\,_0^{\max(\ell,1)}(\R^d))^2$, we have
$$
\Vert \varphi_1\cdot q-\varphi_2\cdot q\Vert\leq\gamma\Vert(\varphi_1-\varphi_2)_{\vert U}\Vert_{\ell},
$$
where $(\varphi_1-\varphi_2)_{\vert U}$ denotes the restriction of $\varphi_1-\varphi_2$ to $U$.
\end{definition}

Except for $\text{Diff}\,_0^{\max(\ell,1)}(\R^d)$ itself, every state of every shape space given so far in examples had compact support.

\begin{theorem}\label{existence}
Assume that $V$ is an RKHS of vector fields of class $\mathcal{C}^{\ell+1}$ on $\R^d$, that $q\mapsto C_q$ is continuous, and that $g$ is bounded below and lower semi-continuous. If $q_0$ has compact support, then Problem \ref{diffprob1} has at least one solution.
\end{theorem}

%The representation of shape analysis problems  is however not completely satisfying from the practical point of view. 
In practice one does not usually have available a convenient, functional definition of the space $V$ of vector fields. The RKHS $V$ is in general only known through its kernel $K$, as already mentioned in Section \ref{sec2.1} (and the kernel is often a Gaussian one). Hence Problem \ref{diffprob1}, formulated as such, is not easily tractable since one might not have a good knowledge (say, a parametrization) of the space $V$.

One can however derive, under a slight additional assumption, a different formulation of Problem \ref{diffprob1} that may be more convenient and appropriate in view of practical issues. This is done in the next section, in which our aim is to obtain an optimal control problem only depending on the knowledge of the reproducing kernel $K$ of the space $V$ (and not directly on $V$ itself), the solutions of which can be lifted back to the group of diffeomorphisms.

\paragraph{Kernel formulation of the optimal control problem.}
For a given $q\in M$, consider the transpose $\xi_q^*:X^*\rightarrow V^*$ of the continuous linear mapping $\xi_q:V\rightarrow X$. This means that for every $u\in X^*$ the element $\xi_q^*u\in V^*$ (sometimes called pullback) is defined by $\langle \xi_q^*u,v\rangle_{V^*,V}=\langle u,\xi_q(v)\rangle_{X^*,X}$, for every $v\in V$.
Besides, by definition of $K_V$, there holds $\langle \xi_q^*u,v\rangle_{V^*,V}=(K_V\xi_q^*u,v)_V$.
The mapping $(q,u)\in M\times X^*\mapsto \xi_q^*u\in V^*$ is often called the \textit{momentum map} in control theory \cite{MRT}.

We start our discussion with the following remark. As seen in Example \ref{example1}, we observe that, in general, given $q\in M$ the mapping $\xi_q$ is far from being injective (i.e., one-to-one). Its null space $\ker(\xi_q)$ can indeed be quite large, with many possible time-dependent vector fields $v$  generating the same solution of $q(0)=q_0$ and $\dot{q}(t)=\xi_{q(t)} v(t)$ for almost every $t\in[0,1]$.

A usual way to address this overdetermination consists of selecting, at every time $t$, a $v(t)$ that has minimal norm subject to $\xi_{q(t)} v(t)=\dot{q}(t)$ (resulting in a least-squares problem).
This is the object of the following lemma.

\begin{lemma}\label{lemmemoindrecarre}
Let $q\in M$. Assume that $\Img(\xi_q)=\xi_q(V)$ is closed. Then, for every $v\in V$ there exists $u\in X^*$ such that $\xi_q v = \xi_q K_V \xi_q^*u$. Moreover, the element $K_V \xi_q^*u \in V$ is the one with minimal norm over all elements $v'\in V$ such that $\xi_q v' =\xi_q v$.
\end{lemma}

\begin{proof}
Let $\hat v$ denote the orthogonal projection of 0 on the space $A = \{v': \xi_q v' = \xi_q v\}$, i.e., the element of $A$ with minimal norm. Then $\hat v$ is characterized by $\xi_q \hat v = \xi_q v$ and $\hat v \in \ker(\xi_q)^\perp$. Using the Banach closed-range theorem, we have $(\ker(\xi_q))^\perp = K_V \Img(\xi_q^*)$, so that there exists $u\in X^*$ such that $\hat v = K_V \xi_q^*u$, and hence $\xi_q v = \xi_q K_V \xi_q^* u$. 
\end{proof}

\begin{remark} Note that the latter assertion in the proof does not require $\Img(\xi_q)$ to be closed, since we always have $K_V(\Img(\xi_q))\subset (\ker(\xi_q)^\perp)$.
\end{remark}

Whether $\Img(\xi_q)=\xi_q(V)$ is closed or not, this lemma and the previous discussion suggest replacing the control $v(t)$ in Problem \ref{diffprob1} by $u(t)\in X^*$ such that $v(t)=K_V \xi_{q(t)}^*u(t)$. Plugging this expression into the system $\dot{q}(t)=\xi_{q(t)} v(t)$ leads to the new control system $\dot{q}(t)=K_{q(t)} u(t)$, where
\begin{equation}\label{Kq}
K_q=\xi_qK_V\xi_q^*,
\end{equation}
for every $q\in M$.
The operator $K_q:X^*\rightarrow X$ is continuous and symmetric (i.e., $\langle u_2,K_qu_1\rangle_{X^*,X}=\langle u_1,K_qu_2\rangle_{X^*,X}$ for all $(u_1,u_2)\in(X^*)^2$), satisfies $\langle u, K_qu\rangle_{X^*,X}=\Vert K_V\xi_q^*u\Vert_V^2$ for every $u\in X$ and thus is positive semi-definite, and $(q,u)\mapsto K_qu$ is as regular as $(q,v)\mapsto \xi_qv$. 
Note that $K_q(X^*)=\xi_q(V)$ whenever $\xi_q(V)$ is closed.

This change of variable appears to be particularly relevant since the operator $K_q$ is usually easy to compute from the reproducing kernel $K_V$ of $V$, as shown in the following examples.

\begin{example}
Let $M=X=\mathcal{C}^0(S,\R^d)$ be the set of continuous mappings from a Riemannian manifold $S$ to $\R^d$. The action of $\text{Diff}^0(\R^d)$ is smooth of order $0$, with $\xi_qv=v\circ q$ (see Example \ref{example1}).
Let $V$ be an RKHS of vector fields of class $\mathcal{C}^1_0$ on $\R^d$, with reproducing kernel $K$. 
Every $u\in X^*$ can be identified with a vector-valued Radon measure on $S$. Then
$$
\langle \xi_q^*u, v\rangle_{V^*,V}=\langle u, v\circ q\rangle_{X^*,X}= \int_Sv(q(s))^Tdu(s),
$$
for every $q\in M$ and for every $v\in V$. In other words, one has
$\xi_q^*u = \int_S du(s)\otimes\delta_{q(s)}$, and therefore, by definition of the kernel, we have $K_V\xi_q^*u = \int_S K_V(du(s)\otimes\delta_{q(s)}) = \int_S K(\cdot,q(s))\, du(s)$. We finally infer that
$$
K_qu(t)=\int_S K(q(t),q(s))\, du(s).
$$
\end{example}

\begin{example}
Let $X=(\R^d)^n$ and $M=\mathrm{Lmk}_d(n)$ (as in Example \ref{example1}).
Then $\xi_{q}v=(v(x_1),\dots,v(x_n))$, and every $u=(u_1,\dots,u_n)$ is identified with a vector of $X$ by $\langle u, w\rangle_{X^*,X}=\sum_{j=1}^nu_j^Tw_j$. Therefore, we get $\xi_q^*u=\sum_{j=1}^nu_j\otimes\delta_{x_j}$, and
$K_V\xi_q^*u=\sum_{j=1}^nK(x_j,\cdot)u_j$. It follows that 
$$K_qu=\left(\sum_{j=1}^nK(x_1,x_j)u_j,\sum_{j=1}^nK(x_2,x_j)u_j,\dots,\sum_{j=1}^nK(x_n,x_j)u_j \right) .$$
In other words, $K_q$ can be identified with matrix of total size $nd\times nd$ and made of square block matrices of size $d$, with the block $(i,j)$ given by $K(x_i,x_j)$.
\end{example}

Following the discussion above and the change of control variable $v(t)=K_V \xi_{q(t)}^*u(t)$, we are led to consider the following optimal control problem.

\begin{problem}\label{diffprob2}
Let $q_0\in M$, and let $C:M\times V\rightarrow Y$ be a mapping such that $v\mapsto C(q,v)=C_qv$ is linear for every $q\in M$. 
Let $g:M\rightarrow\R$ be a function.
We consider the problem of minimizing the functional
\begin{equation}\label{defJ2}
J_2(q,u)=\frac{1}{2}\int_0^1\langle u(t), K_{q(t)}u(t)\rangle_{X^*,X}dt+g(q(1))=E(q(t))+g(q(1))
\end{equation}
over all couples $(q(\cdot),u(\cdot))$, where $u:[0,1]\rightarrow X^*$ is a measurable function and $q(\cdot)\in W^{1,1}_{q_0}(0,1;M)$ are such that $\dot{q}(t)=K_{q(t)} u(t)$ and $C_{q(t)}K_V\xi_{q(t)}^*u(t)=0$ for almost every $t\in [0,1]$.
\end{problem}

The precise relation between both problems is clarified in the following result.

\begin{proposition}\label{proposition1}
Assume that $\ker(\xi_q)\subset\ker (C_q)$ and that $\Img(\xi_q)=\xi_q(V)$ is closed, for every $q\in M$.
Then Problems \ref{diffprob1} and \ref{diffprob2} are equivalent in the sense that $\inf J_1 = \inf J_2$ over their respective sets of constraints.

Moreover, if $(\bar q(\cdot),\bar u(\cdot))$ is an optimal solution of Problem \ref{diffprob2}, then $(\bar q(\cdot),\bar v(\cdot))$ is an optimal solution of Problem \ref{diffprob1}, with  $\bar v(\cdot)=K_V\xi^*_{\bar q(\cdot)}\bar u(\cdot)$ and $\bar q(\cdot)$ the corresponding curve defined by $\bar q(0)=q_0$ and $\dot{\bar q}(t)=K_{\bar q(t)} \bar u(t)$ for almost every $t\in[0,1]$.
Conversely, if $(\bar q(\cdot),\bar v(\cdot))$ is an optimal solution of Problem \ref{diffprob1} then there exists a measurable function $\bar u:[0,1]\rightarrow X^*$ such that $\bar v(\cdot)=K_V\xi^*_{\bar q(\cdot)}\bar u(\cdot)$, and $\bar u(\cdot)$ and $(\bar q(\cdot),\bar u(\cdot))$ is an optimal solution of Problem \ref{diffprob2}.
\end{proposition}

\begin{proof}
First of all, if $J_2(q,u)$ is finite, then $v(\cdot)$ defined by $v(t)=K_V\xi_{q(t)}u(t)$ belongs to $L^2(0,1;V)$ and therefore, using the differential equation $\dot{q}(t)=\xi_{q(t)}v(t)$ for almost every $t$ and the Gronwall lemma, we infer that $q\in H^1_{q_0}(0,1;M)$. The inequality $\inf J_1 \leq \inf J_2$ follows obviously.

Let us prove the converse. Let $\varepsilon>0$ arbitrary, and let $v(\cdot)\in L^2(0,1;V)$ and $q\in H^1_{q_0}(0,1;M)$ be such that $J_1(q,v)\leq \inf J_1+\varepsilon$, with $\dot{q}(t)=\xi_{q(t)}v(t)$ and $C_{q(t)}v(t)=0$ for almost every $t\in[0,1]$. We can write $v(t)=v_1(t)+v_2(t)$ with $v_1(t)\in\ker(\xi_{q(t)})$ and $v_2(t)\in (\ker(\xi_{q(t)}))^\perp=\Img(K_V\xi_{q(t)}^*)$, for almost every $t\in[0,1]$, with $v_1(\cdot)$ and $v_2(\cdot)$ measurable functions, and obviously one has $\int_0^T\Vert v_2(t)\Vert_V^2\, dt \leq \int_0^T\Vert v(t)\Vert_V^2\, dt$. Then, choosing $u(\cdot)$ such that $v_2(\cdot)=K_V\xi_{q(\cdot)}^* u(\cdot)$, it follows that $J_2(u)=J_1(v_2)\leq J_1(v)\leq\inf J_1+\varepsilon$. Therefore $\inf J_2\leq\inf J_1$. The rest is obvious.
%The proof easily follows from the fact that $\overline{\mathrm{Im}\,\xi_q} = (\ker K_V\xi_q^*)^\perp = (\ker \xi_qK_V\xi_q^*)^\perp = \overline{\mathrm{Im}\,\xi_q K_V \xi_q^*}$.
\end{proof}

\begin{remark}
Under the assumptions of Proposition \ref{proposition1} and of Theorem \ref{existence}, Problem \ref{diffprob2} has at least one solution $\bar u(\cdot)$, there holds $\min J_1 = \min J_2$, and the minimizers of Problems \ref{diffprob1} and \ref{diffprob2} are in one-to-one correspondance according to the above statement.
\end{remark}

\begin{remark}
The assumption $\ker(\xi_q)\subset\ker (C_q)$ is satisfied in the important case where the constraints are kinetic, and is natural to be considered since it means that, in the problem of overdetermination in $v$, the constraints can be passed to the quotient (see Lemma \ref{lemmemoindrecarre}).
Actually for kinetic constraints we have the following interesting result (proved further, see Remark \ref{rem_propequiv}), completing the discussion on the equivalence between both problems.
\end{remark}

\begin{proposition}\label{equiv}
Assume that $V$ is an RKHS of vector fields of class at least $\mathcal{C}^{\ell+1}_0$ on $\R^d$, that the constraints are kinetic, i.e., are of the form $C_q\xi_qv=0$, and that the mapping $(q,w)\mapsto C_qw$ is of class $\mathcal{C}^1$. 
If $C_q\xi_q$ is surjective (onto) for every $q\in M$, then for every optimal solution $\bar v$ of Problem \ref{diffprob1} there exists a measurable function $\bar u:[0,1]\rightarrow X^*$ such that $\bar v=K_V\xi^*_{\bar q}\bar u$, and $\bar u$ is an optimal solution of Problem \ref{diffprob2}.
\end{proposition}

Note that this result does not require the assumption that $\Img(\xi_q)=\xi_q(V)$ be closed.

\begin{remark}
It may happen that Problems \ref{diffprob1} and \ref{diffprob2} do not coincide whenever $\Img(\xi_q)$ is not closed.
Actually, if the assumption that $\Img(\xi_q)$ is closed is not satisfied then it may happen  that the set of controls satisfying the constraints in Problem \ref{diffprob2} be reduced to the zero control. 

Let us provide a situation where this occurs. Let $v(q)\in \overline{\Img(K_V\xi_q^*)}\setminus\Img(K_V\xi_q^*)$ with $\Vert v(q)\Vert_V=1$. In particular $v(q)\in (\ker(\xi_q))^\perp$.
Assume that $C_q$ is defined as the orthogonal projection onto $(\R v(q)\oplus\ker(\xi_q))^\perp=v(q)^\perp\cap(\ker(\xi_q))^\perp$. Then $\ker (C_q) = \R v(q)\oplus\ker(\xi_q)$.
We claim that $\ker (C_qK_V\xi_q^*)=\{0\}$. Indeed, let $u\in X^*$ be such that $C_qK_V\xi_q^*u=0$. Then on the one part $K_V\xi_q^*u\in \ker (C_q)$, and on the other part, $K_V\xi_q^*u\in\Img(K_V\xi_q^*)\subset (\ker(\xi_q))^\perp$. Therefore $K_V\xi_q^*u\in \ker (C_q)\cap (\ker(\xi_q))^\perp=\R v(q)$, but since $v(q)\notin \Img(K_V\xi_q^*)$, necessarily $u=0$.
\end{remark}

\subsection{Further comments: lifted shape spaces and multishapes}\label{mlmk}
In this section we provide one last way to study shape spaces and describe two interesting and important variants of shape spaces, namely lifted shape spaces and multishapes. We show that a slightly different optimal control problem can  model the shape deformation problem in these spaces.

\paragraph{Lifted shape spaces.}
Lifted shapes can be used to keep track of additional parameters when studying the deformation of a shape. For example, when studying $n$ landmarks $(x_1,\dots,x_n)$ in $\R^d$, it can be interesting to keep track of how another point $x$ is moved by the deformation.

Let $M$ and $\hat{M}$ be two shape spaces, open subsets of two Banach spaces $X$ and $\hat{X}$ respectively, on which the group of diffeomorphisms of $\R^d$ acts smoothly with respective orders $\ell$ and $\hat{\ell}$. Let $V$ be an RKHS of vector fields in $\R^d$ of class $\mathcal{C}_0^{\max(\ell,\hat{\ell})}$. We denote by $\xi_q$ (respectively $\xi_{\hat{q}}$) the infinitesimal action of $V$ on $M$ (respectively $\hat{M}$).
We assume that there exists a $\mathcal{C}^1$ equivariant submersion $P:\hat{M}\rightarrow M$. 

By equivariant, we mean that $P(\varphi\cdot\hat{q})=\varphi\cdot P(\hat{q})$, for every diffeomorphism $\varphi\in \text{Diff}\, ^{\max(\ell,\hat{\ell}),1}$ and every $\hat{q}\in\hat{M}$. Note that this implies that $dP_{\hat{q}}.\xi_{\hat{q}}=\xi_{q}$ and $\xi_{\hat{q}}^*dP_{\hat{q}}^*=\xi_{q}^*.$

For example, for $n<\hat{n}$, the  projection $P:\mathrm{Lmk}_d(\hat{n})\rightarrow \mathrm{Lmk}_d(n)$ defined by $P(x_1,\dots,x_{\hat{n}})=(x_1,\dots,x_n)$ is a $\mathcal{C}^1$ equivariant submersion. 
More generally, for a compact Riemannian manifold $\hat{S}$ and a submanifold $S\subset \hat{S}$, the restriction mapping $P:\mathrm{Emb}(\hat{S},\R^d)\rightarrow \mathrm{Emb}(S,\R^d)$ defined by $P(q)=q_{\vert S}$ is a $\mathcal{C}^1$ equivariant submersion for the action by composition of $\text{Diff}\,^1(\R^d)$.

\medskip

The constructions and results of Section \ref{modshsp} can be applied to this setting, and in particular the deformation evolution induces a control system on $\hat{M}$, as investigated previously. 

\begin{remark}\label{refapprox}
Let $V$ be an RKHS of bounded vector fields of class $C_0^{\max(\ell,\hat{\ell})+1}$. Let $g$ be a data attachment function on $M$ and let $C$ be a mapping of constraints. We set $\hat{g}=g\circ P$ and $\hat{C}_{\hat{q}}=C_{P(\hat{q})}$. Then a time-dependent vector field $v$ in $V$ is a solution of Problem \ref{diffprob1} for $M$ with constraints $C$ and data attachment $g$ if and only if it is also a solution of Problem \ref{diffprob1} for $\hat{M}$ with constraints $\hat{C}$ and data attachment $\hat{g}$. This remark will be used for finite-dimensional approximations in Section \ref{approx}.
\end{remark}

One can however define a control system of a different form, by lifting the control applied on the smaller shape space $M$ to the bigger shape space $\hat{M}$. 

The method goes as follows. Let $q_0\in M$ and $\hat{q}_0\in P^{-1}(q_0)$. Consider a measurable map $u:[0,1]\rightarrow X^*$ and the corresponding curve $q(\cdot)$ defined  by $q(0)=q_0$ and $\dot{q}(t)=K_{q(t)}u(t)$ for almost every $t\in[0,1]$, where $K_q=\xi_qK_V\xi_q^*$. This curve is the same as the one induced by the time-dependent vector field $v(\cdot)=K_V\xi_{q(\cdot)}^*u(\cdot)$.
The deformation $\varphi$ corresponding to the flow of $v$ defines on $\hat{M}$ a new curve $\hat{q}(t)=\varphi(t)\cdot \hat{q}_0$ with speed
$$
\dot{\hat{q}}(t)=\xi_{\hat{q}}K_V\xi_{q(t)}^*u(t)=K_{\hat{q}(t)}dP_{\hat{q}(t)}^*u(t),
$$
with $K_{\hat{q}}=\xi_{\hat{q}}K_V\xi^*_{\hat{q}}$. Note that $P(\hat{q}(t))=q(t)$ for every $t\in[0,1]$.
We have thus obtained a new class of control problems.

\begin{problem}\label{liftprobmo}
Let $\hat{q}_0\in\hat{M}$, and let $C:\hat{M}\times V\rightarrow Y$ be continuous and linear with respect to the second variable, with $Y$ a Banach space.
Let $g:\hat{M}\rightarrow\R$ be a real function on $\hat{M}$.  
We consider the problem of minimizing the functional
$$
J_3(\hat{q},u)=\frac{1}{2}\int_0^1\langle u(t),K_{P(\hat{q}(t))}u(t)\rangle_{X^*,X} dt+g(\hat{q}(1))
$$
over all $(\hat{q}(\cdot),u(\cdot))$, where $u:[0,1]\rightarrow X^*$ is a measurable function and $\hat{q}(\cdot)\in W^{1,1}_{\hat{q}_0}(0,1;\hat{M})$ are such that $\dot{\hat{q}}(t)=K_{\hat{q}(t)}dP_{\hat{q}(t)}^*u(t)$ and 
$C_{\hat{q}(t)}K_V\xi_{P(\hat{q}(t))}^*u(t)=0$ for almost every $t\in[0,1]$.
\end{problem}

Note that, if $g$ and $C$ only depend on $P(\hat{q})$ then the solutions $u(t)$ of Problem \ref{liftprobmo} coincide with the ones of Problem \ref{diffprob2} on $M$.

Problem \ref{liftprobmo} can be reformulated back into an optimal control problem on $V$ and on $\hat{M}$, similar to Problem \ref{diffprob1}, by adding the constraints $D_{\hat{q}}v=0$ where $D_{\hat{q}}v$ is the orthogonal projection of $v$ on $\ker (\xi_{P(\hat{q})})$.

Some examples of lifted shape spaces can be found in \cite{Y1}, where controls are used from a small number of landmarks to match a large number of landmarks, with additional state variables defining Gaussian volume elements. 
%The corresponding equivariant submersion $P$ is the natural projection $\mathrm{Lmk}_d(\hat{n})\rightarrow \mathrm{Lmk}_d(n)$, with $n<\hat{n}$.
Another application of lifted shape spaces will be mentioned in Section \ref{approx}, where they will be used to approximate infinite-dimensional shape spaces by finite-dimensional ones.

\paragraph{Multishapes.}

Shape analysis problems sometimes involve collections of shapes that must be studied together, each of them with specific properties associated with a different space of vector fields. These situations can be modeled as follows. 

Consider some shape spaces $M_1,\dots,M_k$, open subsets of Banach spaces $X_1,\dots,X_k,$ respectively, on which diffeomorphisms of $\R^d$ acts smoothly on each shape space $M_i$ with order $\ell_i$. Let $k_i\geq 1$, and consider $V_1,\dots,V_k$, RKHS's of vector fields of $\R^d$ respectively of class $\mathcal{C}^{\ell_i+k_i}_0 $ with kernels $K_1,\dots,K_k$, as defined in Section \ref{sec2.1}. 
In such a model we thus get $k$ control systems, of the form $\dot{q}_i(\cdot)=\xi_{i,q_i(\cdot)}v_i(\cdot)$, with the controls $v_i(\cdot)\in L^2(0,1;V_i)$, $i=1,\ldots, k$.
The shape space of a multi-shape is a space of the form $M=M_1\times\dots\times M_k$. Let $q_0=(q_{1,0},\dots,q_{k,0})\in M$. Similarly to the previous section we consider the problem of minimizing the functional
$$
\sum_{i=1}^k\int_0^1\Vert v_i(t)\Vert_{V_i}^2dt +g(q_1(1),\dots,q_k(1)),
$$
over all time-dependent vector fields $v_i(\cdot)\in L^2(0,1;V_i)$, $i=1,\ldots, k$, and with
$q_i(1)=\varphi_i(1)\cdot q_{i,0}$ where $\varphi_i$ is the flow generated by $v_i$ (note that, here, the problem is written without constraint).

As in Section \ref{modshsp}, the kernel formulation of this optimal control problem consists of minimizing the functional
\begin{equation}\label{deffunctionalmultishapes}
\frac{1}{2}\int_0^1\sum_{i=1}^k \langle u_i(t), K_{q_i(t),i}u_i(t)\rangle_{X_i^*,X_i} \, dt+g(q(1)).
\end{equation}
over all measurable functions $u(\cdot)=(u_1(\cdot),\dots,u_k(\cdot))\in L^2(0,1;X_1^*\times\dots\times X_k^*)$, where  the curve $q(\cdot)=(q_1(\cdot),\dots,q_k(\cdot)):[0,1]\rightarrow M$ is the solution of $q(0)=q_0$ and
\begin{equation}\label{contsysmultishape}
\dot{q}(t)=K_{q(t)}u(t)=\left(K_{1,q_1(t)}u_1(t),\dots,K_{k,q_k(t)}u_k(t)\right),
\end{equation}
for almost every $t\in[0,1]$, with $K_{i,q_i}=\xi_{i,q_i}K_{V_i}\xi_{i,q_i}^*$ for $i=1,\ldots,k$. 

Obviously, without any further consideration, studying this space essentially amounts to studying each $M_i$ separately, the only interaction possibly intervening from the final cost function $g$. More interesting problems arise however when the shapes can interact with each other, and are subject to consistency constraints. For example, assume that one studies a cat shape, decomposed into two parts for the body and the tail. Since these parts have very different properties, it makes sense to consider them \textit{a priori} as two distinct shapes $S_1$ and $S_2$, with shape spaces $M_1=\mathcal{C}^0(S_1,\R^3)$ and $M_2=\mathcal{C}^0(S_2,\R^3)$, each of them being associated with RKHS's $V_1$ and $V_2$ respectively.
Then, in order to take account for the tail being attached to the cat's body, the contact point of the body and the tail of the cat must belong to both shapes and be equal. In other words, if $q_1\in M_1$ represents the body and $q_2\in M_2$ the tail, then there must hold $q_1(s_1)=q_{2}(s_2)$ for some $s_1\in S_1$ and $s_2\in S_2$. This is a particular case of state constraints, i.e., constraints depending only on the state $q$ of the trajectory.

Considering a more complicated example, assume that two (or more) shapes are embedded in a given background. Consider two states $q_1$ and $q_2$ in respective  spaces $M_1=\mathcal{C}^0(S_1,\R^d)$ and $M_2=\mathcal{C}^0(S_2,\R^d)$ of $\R^d$. Assume that they represent the boundaries of two disjoint open subsets $U_1$ and $U_2$ of $\R^d$. We define a third space  $M_3=M_1\times M_2$, whose elements are of the form $q_3=(q_{3}^1,q_{3}^2)$. This shape space represents the boundary of the complement of $U_1\cup U_2$ (this complement being the background). Each of these three shape spaces is acted upon by the diffeomorphisms of $\R^d$. Consider for every $M_i$ an RKHS $V_i$ of vector fields.
The total shape space is then $M=M_1\times M_2\times M_3$, an element of which is given by $q=(q_1,q_2,q_3)=(q_1,q_2,q_{3}^1,q_{3}^2)$. Note that $\partial (U_1\cup U_2)=\partial (U_1\cup U_2)^c$. However, since $(q_1,q_2)$ represents the left-hand side of this equality, and $q_3=(q_{3}^1,q_{3}^1)$ the right-hand side, it only makes sense to impose the constraints
$q_1=q_{3}^1$ and $q_2=q_{3}^2$.
This model can be used for instance to study two different shapes that are required not to overlap during the deformation.

In this example, one can even go further: the background does not need to completely mimic the movements of the shapes. We can for example let the boundaries slide on each another. This imposes constraints on the speed of the shapes (and not just on the shapes themselves), of the form $C_qK_qu=0$. See section \ref{sec7} for additional details.

Multi-shapes are of great interest in computational anatomy and provide an important motivation to study shape deformation under constraints.

\subsection{Finite dimensional approximation of optimal controls}\label{approx}
The purpose of this section is to show that at least one solution of Problem \ref{diffprob1} can be approximated by a sequence of solutions of a family of nested optimal control problems on finite-dimensional shape spaces with finite-dimensional constraints. We assume throughout that $\ell\geq 1$.

Let $(Y^n)_{n\in\N}$ be a sequence of Banach spaces and $(C^n)_{n\in\N}$ be a sequence of continuous mappings $C^n:M\times V \rightarrow Y^n$ that are linear and continuous with respect to the second variable.
Let $(g^n)_{n\in\N}$ be a sequence of continuous functions on $M$, bounded from below with a constant independent of $n$. For every integer $n$, we consider the problem of minimizing the functional
$$
J^n_1(v)=\frac{1}{2}\int_0^1\Vert v(t)\Vert_V^2dt+g^n(q(1)),
$$
over all $v(\cdot)\in L^2(0,1;V)$ such that $C^n_{q(t)}v(t)=0$ for almost every $t\in[0,1]$, where $q(\cdot):[0,1]\rightarrow M$ is the curve defined by $q(0)=q_0$ and $\dot{q}(t)=\xi_{q(t)} v(t)$ for almost every $t\in[0,1]$.
It follows from Theorem \ref{existence} that there exists an optimal solution $v^n(\cdot)\in L^2(0,1;V)$. We denote by $q^n(\cdot)$ the corresponding curve.

\begin{proposition}\label{apdimfinie}
Assume that $V$ is an RKHS of vector fields of class $\mathcal{C}^{\ell+1}_0$ on $\R^d$ and that the sequence $(\ker (C^n_q))_{n\in\N}$ is decreasing (in the sense of the inclusion) and satisfies 
$$ 
\bigcap_{n\in \N} \ker (C_q^n)=\ker(C_q),
$$
for every $q\in M$. Assume that $g^n$ converges to $g$ uniformly on every compact subset of $M$. Finally, assume that $q_0$ has compact support.
Then the sequence $(v^n(\cdot))_{n\in\N}$ is bounded in $L^2(0,1;V)$, and every cluster point of this sequence for the weak topology of $L^2(0,1;V)$ is an optimal solution of Problem \ref{diffprob1}. 
More precisely, for every cluster point $\bar v(\cdot)$ of $(v^n(\cdot))_{n\in\N}$, there exists a subsequence such that $(v^{n_j}(\cdot))_{j\in\N}$ converges weakly to $\bar v(\cdot)\in L^2(0,1;V)$, the sequence $(q^{n_j}(\cdot))_{j\in\N}$ of corresponding curves converges uniformly to $\bar q(\cdot)$, and $J_1^{n_j}(v^{n_j})$ converges to $\min J_1=J_1(\bar v)$ as $j$ tends to $+\infty$, and $\bar v(\cdot)$ is a solution of Problem \ref{diffprob1}.
\end{proposition}

\begin{proof}
The sequence $(v^n(\cdot))_{n\in\N}$ is bounded in $L^2(0,1;V)$ as a consequence of the fact that the functions $g^n$ are uniformly bounded below. Let $\bar v(\cdot)$ be a cluster point of this sequence for the weak topology of $L^2(0,1;V)$.
Assume that $(v^{n_j}(\cdot))_{j\in\N}$ converges weakly to $\bar v(\cdot)\in L^2(0,1;V)$. Denoting by $\bar q(\cdot)$ the curve corresponding to $\bar v(\cdot)$, the sequence $(q^{n_j}(\cdot))_{j\in\N}$ converges uniformly to $\bar q(\cdot)$ (see Lemma \ref{seqcomp}).
Using the property of decreasing inclusion, we have $C^N_{q^{n_j}(\cdot)}v^{n_j}(\cdot)=0$ for every integer $N$ and every integer $j\geq N$.
Using the same arguments as in the proof of Theorem \ref{existence} (see Section \ref{sec_proof_thm_existence}), it follows that $C_{\bar q(\cdot)}\bar v(\cdot)=0$. Finally, since $\int_0^1\Vert \bar v(t)\Vert_V^2 \, dt\leq \liminf\int_0^1\Vert v^n(t)\Vert_V^2\, dt$, and since  $g^n$ converges uniformly to $g$ on every compact subset of $M$, it follows that $J_1(\bar v)\leq\liminf J_1^{n_j}(v^{n_j})$.

%We first prove that $\limsup J_1^n(v^n)\leq \min J_1$.
Since every $v\in\ker (C_q)$ belongs as well to $\ker (C^{n_j}_q)$, it follows that $J_1^{n_j}(v^{n_j})\leq J_1^{n_j}(v)$, for every time-dependent vector field $v(\cdot)\in L^2(0,1;V)$ such that $C_{q(\cdot)}v(\cdot)=0$, where $q(\cdot):[0,1]\rightarrow M$ is the curve corresponding to $v(\cdot)$. Since $g^n$ converges uniformly to $g$, one has $J_1^{n_j}(v)\rightarrow J_1(v)$ as $n\rightarrow +\infty$. %The claimed inequality follows.
It follows that $\limsup J_1^{n_j}(v^{n_j})\leq \min J_1$.

We have proved that $J_1(\bar v)\leq\liminf J_1^{n_j}(v^{n_j})\leq \limsup J_1^{n_j}(v^{n_j})\leq \min J_1$, and therefore $J_1(\bar v)=\min J_1$, that is, $\bar v(\cdot)$ is an optimal solution of Problem \ref{diffprob1}, and $J_1^{n_j}(v^{n_j})$ converges to $\min J_1=J_1(\bar v)$ as $j$ tends to $+\infty$.
\end{proof}

\paragraph{Application: approximation with finite dimensional shape spaces.}
Let $S^1$ be the unit circle of $\R^2$, let $\ell\geq 1$ be an integer, $X=\mathcal{C}^\ell(S^1,\R^d)$ and let $M=\mathrm{Emb}^\ell(S^1,\R^d)$ be the space of parametrized simple closed curves of class $\mathcal{C}^\ell$ on $\R^d$. We identify $X$ with the space of all mappings $f\in \mathcal{C}^\ell([0,1],\R^d)$ such that $f(0)=f(1)$, $f'(0)=f'(1)$, ..., $f^{(\ell)}(0)=f^{(\ell)}(1)$.
The action of the group of diffeomorphisms of $\R^d$ on $M$, defined by composition, is smooth of order $\ell$ (see Section \ref{modshsp}).
Let $q_1\in M$ and $c>0$ fixed. We define $g$ by
$$
g(q)=c\int_{S^1} \vert q(t)-q_1(t)\vert^2 \, dt,
$$
for every $q\in M$.
Consider pointwise kinetic constraints $C:M\times X\rightarrow \mathcal{C}^0(S^1,\R^m)$, defined by $(C_q\dot{q})(s)=F_{q(s)}\dot{q}(s)$ for every $s\in S^1$, with $F\in \mathcal{C}^0(\R^d,\mathcal{M}_{m,d}(\R))$, where $\mathcal{M}_{m,d}(\R)$ is the set of real matrices of size $m\times d$.

Note that the multishapes constraints described in Section \ref{mlmk} are of this form.

Our objective is to approximate this optimal control problem with a sequence of optimal control problems specified on the finite dimensional shape spaces $\mathrm{Lmk}_d(2^n)$.

For $n\in \N$, let $\xi^n$ be the infinitesimal action of the group of diffeomorphisms on $\mathrm{Lmk}_d(2^n)$, the elements of which are denoted by $q^n=(x_1^n,\dots,x_{2^n}^n)$. Define on the associated control problem the kinetic constraints $\tilde{C}^n_{q^n}\xi^n_{q_n}v=(F_{x_1^n}v(x_1^n),\dots,F_{x_{2^n}^n}v(x_{2^n}^n))$,
and the data attachment function
$$
\tilde{g}^n(q^n)=\frac{c}{2^n}\sum_{r=1}^{2^n}\vert x_r^n-q_1(2^{-n}r)\vert.
$$
Let $v^n(\cdot)$ be an optimal control of Problem \ref{diffprob1} for the above optimal control problem specified on $\mathrm{Lmk}_d(2^n)$.

\begin{proposition}\label{findimapprox}
Every cluster point of the sequence $(v^n(\cdot))_{n\in\N}$ for the weak topology on $L^2(0,1;V)$ is a solution of Problem \ref{diffprob1} specified on $M=\mathrm{Emb}^\ell(S^1,\R^d)$ with constraints and minimization functional respectively given by $C$ and $g$ defined above.
\end{proposition}

\begin{proof}
Define the submersions $P^n:M\rightarrow \mathrm{Lmk}_d(2^n)$
by
$$
P^n(q)=(q(2^{-n}),\dots,q(2^{-n}r),\dots,q(1)).
$$
Let $g^n=\tilde{g}^n\circ P^n$ and $C^n_q=\tilde{C}_{P^n(q)}dP_{q}$.
In other words, $g^n$ (resp. $C^n$) are the lifts of $\tilde{g}^n$ (resp. $\tilde{C}^n$) from $\mathrm{Lmk}_d(2^n)$ to $M$ through $P^n$. 
Using Remark \ref{refapprox} on lifted shape spaces, we infer that the optimal control $v_n(\cdot)$ of Problem \ref{diffprob1} specified on the finite dimensional space $\mathrm{Lmk}_d(2^n)$ with constraints $\tilde{C}^n$ and data attachment $\tilde{g}^n$ is also optimal for Problem \ref{diffprob1} specified on the infinite dimensional set $M$ with constraints $C^n$ and data attachment $g^n$.
Now, if $v\in \ker (C^n)$ for every integer $n$, then $F_{q(s)}v(q(s))=0$ for every $s=2^{-n}r$ with $n\in \N$ and $r\in\lbrace 1,\dots,2^n\rbrace$. The set of such $s$ is dense in $[0,1]$, and $s\mapsto F_{q(s)}v(q(s))$ is continuous. Therefore $F_{q(s)}v(q(s))=0$ for every $s\in [0,1]$, that is, $C_{q(\cdot)}\xi_{q(\cdot)}v(\cdot)= 0$. Since the converse is immediate, we get
$\ker (C_{q(\cdot)})=\bigcap_{n\in\N}\ker (C^n_{q(\cdot)})$.
Finally, since $q(\cdot)$ is a closed curve of class at least $\mathcal{C}^\ell$ with $\ell\geq 1$, it is easy to check that
$$
g^n(q)=\frac{c}{2^n}\sum_{r=1}^{2^n}\vert q(2^{-n}r)-q_1(2^{-n}r)\vert
$$
converges to $g$, uniformly on every compact subset of $M$.
Therefore, Proposition \ref{apdimfinie} can be applied to the sequence $(v^n(\cdot))_{n\in\N}$, which completes the proof.
\end{proof}

\begin{remark}\label{remapdimfinie}
The same argument works as well if we replace $S^1$ with any compact Riemannian manifold $\mathcal{S}$, and applies to the vertices of increasingly finer triangulations of $\mathcal{S}$.
\end{remark}

\subsection{Proof of Theorem \ref{existence}}\label{sec_proof_thm_existence}
Let $(v_n(\cdot))_{n\in\N}$ be a sequence of $L^2(0,1;V)$ such that $J_1(v_n)$ converges to its infimum. Let $(\varphi_n)_{n\in\N}$ be the corresponding sequence of deformations and let $(q_n(\cdot))_{n\in\N}$ be the sequence of corresponding curves (one has $q_n(t)=\varphi_n(t)\cdot q_0$ thanks to Remark \ref{rempb1_unknowns}). Since $g$ is bounded below, it follows that the sequence $(v_n(\cdot))_{n\in\N}$ is bounded in $L^2(0,1;V)$.
The following lemma is well known (see \cite{YBOOK}), but we provide a proof for the sake of completeness.

\begin{lemma}\label{seqcomp}
There exist $\bar v(\cdot)\in L^2(0,1;V)$, corresponding to the deformation $\bar\varphi$, and a sequence $({n_j})_{j\in\N}$ of integers such that $(v_{n_j}(\cdot))_{j\in\N}$ converges weakly to  $\bar v(\cdot)$ and such that, for every compact subset $U$ of $\R^d$,
$$
\sup_{t\in[0,1]}\Vert(\varphi_{n_j}(t,\cdot)- \bar\varphi(t,\cdot))_{\vert U}\Vert_{\ell}\underset{j\rightarrow+\infty}{\longrightarrow} 0.
$$
\end{lemma}

\begin{proof}[Proof of Lemma \ref{seqcomp}]
Since the sequence $(v_n(\cdot))_{n\in\N}$ is bounded in the Hilbert  space $L^2(0,1;V)$, there exists a subsequence $(v_{n_j}(\cdot))_{j\in\N}$ converging weakly to some $\bar v\in L^2(0,1;V)$. 
Besides, using \eqref{derdiffeo} for $i=\ell+1$ and the Ascoli theorem, we infer that for every compact subset $U$ of $\R^d$, the sequence $(\varphi_{n_j})_{j\in\N}$ is contained in a compact subset of the space $\mathcal{C}^0([0,1],\mathcal{C}_0^{\ell}(U,\R^d))$.
Considering a compact exhaustion of $\R^d$ and using a diagonal extraction argument, we can therefore extract a subsequence $(\varphi_{n_{j_k}})_{k\in\N}$ with limit $\bar\varphi$ such that, for any compact subset $U$ of $\R^d$,
\begin{equation}\label{cvgf}
\sup_{t\in[0,1]}\Vert(\varphi_{n_{j_k}}(t,\cdot)- \bar\varphi(t,\cdot))_{\vert U}\Vert_{\ell}\underset{k\rightarrow+\infty}{\longrightarrow} 0.
\end{equation}
To complete the proof, it remains to prove that $\bar \varphi$ is the deformation induced by $\bar v(\cdot)$. 
On the first hand, we have $\lim_{k\rightarrow +\infty}\varphi_{n_{j_k}}(t,x)=\bar \varphi(t,x)$, for every $x\in \R^d$ and every $t\in[0,1]$. On the second hand, one has, for every $k\in\N$,
\begin{align*}
& \left\vert\varphi_{n_{j_k}}(t,x)-x-\int_0^t\bar v\circ\bar \varphi(s,x)\, ds\right\vert
=
\left\vert\int_0^t\left(v_{n_{j_k}}\circ\varphi_{n_{j_k}}(s,x)-\bar v\circ\bar \varphi(s,x)\right) ds \right\vert \\
&\leq 
\left\vert\int_0^tv_{n_{j_k}}\circ\varphi_{n_{j_k}}(x)-v_{n_{j_k}}\circ\bar \varphi(s,x) \, ds\right\vert+
\left\vert\int_0^tv_{n_{j_k}}\circ\bar \varphi(s,x)-\bar v\circ\bar \varphi(s,x) \, ds\right\vert.
\end{align*}
Set $C=\sup_{n\in\N}\int_0^1\Vert v_n(t)\Vert_{1} \, dt$, and define $\chi_{[0,t]}\delta_{\bar \varphi(\cdot,x)}\in V^*$ by 
$(\chi_{[0,t]}\delta_{\bar \varphi(\cdot,x)}\vert v)_{L^2(0,1;V)}=\int_0^tv\circ\bar \varphi(s,x) \, ds$.
Then,
\begin{align*}
 \left\vert\varphi_{n_{j_k}}(t,x)-x-\int_0^t\bar v\circ\bar \varphi(s,x)\, ds\right\vert
\leq \ &   C\sup_{s\in[0,t]}\vert\varphi_{n_{j_k}}(s,x)- \bar \varphi(s,x)\vert  \\
& +  \left\vert (\chi_{[0,t]}\delta_{\bar \varphi(\cdot,x)}\vert v_{n_{j_k}}-\bar v)_{L^2(0,1;V)}\right\vert,
\end{align*}
which converges to $0$ as $j$ tends to $+\infty$ thanks to \eqref{cvgf} and to the weak convergence $v_{n_{j_k}}(\cdot)$ to $\bar v(\cdot)$. We thus conclude that $\bar \varphi(t,x)=x+\int_0^t\bar v\circ\bar \varphi(s,x)\, ds$, which completes the proof.
\end{proof}

Setting $\bar q(t)=\bar\varphi(t)\cdot q_0$ for every $t\in[0,1]$, one has $\dot{\bar q}(t)=\xi_{\bar q(t)}\bar v(t)$ for almost every $t\in[0,1]$, and it follows from the above lemma and from the fact that $q_0$ has compact support that 
$$
\displaystyle\sup_{t\in[0,1]}\Vert \bar q(t)-q_{n_j}(t)\Vert_X\underset{j\rightarrow+\infty}{\longrightarrow} 0.
$$

The operator $v(\cdot)\mapsto C_{\bar{q}(\cdot)}v(\cdot)$ is linear and continuous on $L^2(0,1;V)$, so it is also weakly continuous \cite{BBOOK}. We infer that the sequence $C_{q_{n_j}(\cdot)}v_{n_j}(\cdot)$ converges weakly to 
$C_{\bar q(\cdot)}\bar v(\cdot)$ in $L^2(0,1;Y)$. Since $C_{q_{n_j}(\cdot)}v_{n_j}(\cdot)=0$ for every $j\in\N$, it follows that $C_{\bar q(\cdot)}\bar v(\cdot)=0$. In other words, the time-dependent vector field $\bar v(\cdot)$ satisfies the constraints.

It remains to prove that $\bar v(\cdot)$ is indeed optimal. From the weak convergence of the sequence $(v_{n_j}(\cdot))_{j\in\N}$ to $\bar v(\cdot)$ in $L^2(0,1;V)$, we infer that$$\int_0^1\Vert \bar v(t)\Vert_V^2\, dt\leq \liminf_{j\rightarrow+\infty}\int_0^1\Vert v_{n_j}(t)\Vert_V^2\, dt.$$
Besides, since $g$ is lower continuous, $\liminf_{j\rightarrow+\infty} g(q_{n_j}(1))\geq g(\bar q(1))$. Since $J_1(v_{n_j})$ converges to $\inf J_1$, it follows that $J_1(\bar v)=\inf J_1$.

\section{Constrained geodesic equations in shape spaces}\label{secgeod}
In this section, we derive first-order necessary conditions for optimality in Problem \ref{diffprob1}. We extend the well-known Pontryagin maximum principle (PMP) from optimal control theory to our infinite-dimensional framework, under the assumption that the constraints are surjective. This allows us to derive the constrained geodesic equations for shape spaces, and we show how they can be reduced to simple Hamiltonian dynamics on the cotangent space of the shape space.

\subsection{First-order optimality conditions: PMP in shape spaces}\label{PMPs}
%In this section, we derive first-order necessary conditions for optimality, by extending the Pontryagin maximum principle to our setting. 
We address the Pontyagin maximum principle in a slightly extended framework, considering a more general control system and a more general minimization functional than in Problem \ref{diffprob1}.

Let $V$ be a Hilbert space, and let $X$ and $Y$ be Banach spaces. Let $M$ be an open subset of $X$. Let  $\xi:M\times V\rightarrow X$ and $C:M\times V\rightarrow Y$ be mappings of class $\mathcal{C}^1$. Let $L:M\times V \rightarrow \R$ and $g:M\rightarrow\R$ be functions of class $\mathcal{C}^1$. We assume that there exist continuous functions $\gamma_0:\R\rightarrow\R$ and $\gamma_1:X\rightarrow\R$ such that $\gamma_0(0)=0$ and 
\begin{equation}\label{estilag}
\Vert dL_{q',v'}-dL_{q,v}\Vert_{X^*\times V^*}\leq \gamma_0(\Vert q'-q\Vert_X)+\gamma_1(q-q')\Vert v'-v\Vert_V ,
\end{equation}
for all $(q, q')\in M^2$ and all $(v, v')\in V^2$. 

Let $q_0\in M$.
We consider the optimal control problem of minimizing the functional
\begin{equation}\label{defJqv}
J(q,v)=\int_0^1L(q(t),v(t)) \, dt+g(q(1))
\end{equation}
over all $(q(\cdot),v(\cdot))\in H^1_{q_0}(0,1;M)\times L^2(0,1;V)$ such that $\dot{q}(t)=\xi_{q(t)}v(t)$ and $C_{q(t)}v(t)=0$ for almost every $t\in[0,1]$. We define the \textit{Hamiltonian} $H:M\times X^*\times V\times Y^*\rightarrow \R$ by
\begin{equation}\label{defHam}
H(q,p,v,\lambda)=\langle p,\xi_qv\rangle_{X^*,X}-L(q,v)-\langle\lambda, C_qv\rangle_{Y^*,Y}.
\end{equation}
It is a function of class $\mathcal{C}^1$. Using the canonical injection $X\hookrightarrow  X^{**}$, we have $\partial_pH=\xi_qv$.

\begin{remark}
The estimate \eqref{estilag} on $L$ is exactly what is required to ensure that the mapping $(q,v)\mapsto \int_0^1L(q(t),v(t))\, dt$ be well defined and Fr\'echet differentiable for every $(q,v)\in  H_{q_0}^1(0,1;M)\times L^2(0,1;V)$. Indeed, the estimate
$L(q(t),v(t))\leq L(q(t),0)+\gamma_1(q(t),q(t))\Vert v(t)\Vert_V^2$
implies the integrability property. The differentiability is an immediate consequence of the following estimate, obtained by combining \eqref{estilag} with the mean value theorem: for every $t\in [0,1]$, and for some $s_t\in[0,1]$, one has
$$
\begin{aligned}
\vert L(q(t)+\delta q(t),v(t)+\delta v(t))-L(q(t),v(t)-dL_{q(t),v(t)}(\delta q(t),\delta v(t))\vert
\\ 
\leq \Big(\gamma_0(\Vert\delta q(t)\Vert_X)+\gamma_1(s_t\delta q(t))\Vert \delta v(t)\Vert\Big)(\Vert\delta q(t)\Vert_X+\Vert \delta v(t)\Vert_V).
\end{aligned}
$$
\end{remark}

\begin{theorem}\label{PMP}
Assume that the linear operator $C_{q(t)}:V\rightarrow Y$ is surjective for every $q\in M$. %, and that $L$ satisfies \eqref{estilag}. 
Let $(q(\cdot),v(\cdot))\in H_{q_0}^1(0,1;X)\times L^2(0,1;V)$ be an optimal solution of the above optimal control problem. Then there exist $p\in H^1(0,1;X^*)$ and $\lambda\in L^2(0,1;Y^*)$ such that $p(1)+dg_{q(1)}=0$ and
\begin{equation}
\begin{split}
&\dot{q}(t)=\partial_pH(q(t),p(t),v(t),\lambda(t)),\\ 
&\dot{p}(t)=-\partial_qH(q(t),p(t),v(t),\lambda(t)),\\
&\partial_vH(q(t),p(t),v(t),\lambda(t))=0,
\end{split}
\end{equation}
for almost every $t\in[0,1]$.
\end{theorem}

\begin{remark}
This theorem is the extension of the usual PMP to our specific infinite dimensional setting. Any quadruple $q(\cdot),p(\cdot),v(\cdot),\lambda(\cdot)$ solution of the above equations is called an \textit{extremal}.
This is a "weak" maximum principle, in the sense that we derive the condition $\partial_v H$ along any extremal, instead of the stronger maximization condition
$$
H(q(t),p(t),v(t),\lambda(t)) = \max_{w\in \ker (C_{q(t)})} H(q(t),p(t),w,\lambda(t))
$$
for almost every $t\in [0,1]$.
Note however that, in the case of shape spaces, $v\mapsto H(q,p,v)$ is strictly concave and hence both conditions are equivalent.
\end{remark}

\begin{remark}
It is interesting to note that, if we set $V=H^\ell(\R^d,\R^d)$ with $\ell$ large enough, and if $L(q,v)=\frac{1}{2}\int_{\R^d}\vert v(x)\vert^2dx$ and $C_qv=\mathrm{div}\, v$, then the extremal equations given in Theorem \ref{PMP} coincide with the incompressible Euler equation. In other words, we recover the well-known fact that every divergence-free time-dependent vector field minimizing its $L^2$ norm (in $t$ and $x$) must satisfy the incompressible Euler equation (see \cite{Arnold}).
\end{remark}

\begin{remark}
Note that the surjectivity assumption is a strong one in infinite dimension. It is usually not satisfied in the case of shape spaces when $Y$ is infinite dimensional. 
For instance, consider the shape spaces $M=\mathcal{C}^0(S,\R^d)$, with $S$ a smooth compact Riemannian manifold. Let $V$ be an RKHS of vector fields of class $\mathcal{C}^{1}_0$, acting on $M$ as described in Section \ref{modshsp}. 
Let $T$ be a submanifold of $S$ of class $\mathcal{C}^1$. Set $Y=\mathcal{C}^0(T,\R^d)$, and consider the kinetic constraints  $C:M\times V\rightarrow Y$ defined by $C_qv=v\circ q_{\vert T}$. 
If $q$ is differentiable along $T$, then no nondifferentiable map $f\in Y$ is in $\Img( C_q)$.
\end{remark}

\begin{remark}
It is possible to replace $C_q$ with the orthogonal projection on $\ker(C_q)^\perp$, which is automatically surjective. However in this case the constraints become fiber-valued, and for the proof of our theorem to remain valid, one needs to assume that there exists a Hilbert space $V_1$ such that, for every $q_0\in M$, there exists a neighborhood $U$ of $q_0$ in $M$ such that $\bigcup_{q\in U}\ker(C_q)\sim U\times V_1$.
\end{remark}

\begin{remark}
An important consequence of the surjectivity of $C_q$ is that the norm on $Y$ is equivalent to the Hilbert norm induced from $V$ by $C_q$, for every $q\in M$.  The operator $C_qK_VC_q^*:Y^*\rightarrow Y$ is the isometry associated with this norm. In particular $Y$ must be reflexive, and hence $L^2(0,1;Y^*)=L^2(0,1;Y)^*$ (whereas we only have an inclusion for general Banach spaces).
\end{remark}

\begin{remark}
Theorem \ref{PMP} withstands several generalizations. For instance it remains valid whenever we consider nonlinear constraints $C(q,v)=0$ and a general Lagrangian $L$ of class $\mathcal{C}^1$ without any estimates and a nonlinear control system $\dot{q}=\xi(q,v)$, provided that $\xi$ is of class $\mathcal{C}^1$ and that $\partial_vC(q(t),v(t))$ is surjective for every $t\in[0,1]$. However, this requires to consider $v\in L^\infty(0,1;V)$ and makes the proof of the regularity of the Lagrange multipliers slightly more involved.
\end{remark}

Before proving Theorem \ref{PMP}, it can be noted that many versions of the PMP can be found in the existing literature for infinite-dimensional optimal control problems -- for instance, with dynamics consisting of partial differential equations, and however, most of the time, without constraint on the state. Versions of the PMP with state constraints can also be found in the literature (see the survey \cite{HSV}), most of the time in finite dimension, and, for the very few of them existing in infinite dimension, under the additional assumption that the constraints are of finite codimension.
To the best of our knowledge, no version does exist that would cover our specific framework, concerning shape spaces, group actions, with an infinite number of constraints on the acting diffeomorphisms.
The proof that we provide hereafter contains some subtleties such as Lemma \ref{lem4}, and hence Theorem \ref{PMP} is a nontrivial extension of the usual PMP.

\begin{proof}[Proof of Theorem \ref{PMP}]
We define the mapping $\Gamma:H_{q_0}^1(0,1;M)\times L^2(0,1;V)\rightarrow L^2(0,1;X)\times L^2(0,1;Y)$ by $\Gamma(q,v)=(\Gamma_1(q,v),\Gamma_2(q,v))$ %for all $(q,v)\in H_{q_0}^1(0,1;X)\times L^2(0,1;V)$, 
with $\Gamma_1(q,v)(t) = \dot{q}(t)-\xi_{q(t)}v(t)$ and $\Gamma_2(q,v)(t)=C_{q(t)}v(t)$ for almost every $t\in[0,1]$. The mapping $\Gamma$ stands for the constraints imposed to the unknowns of the optimal control problem.

The functional $J:H_{q_0}^1(0,1;M)\times L^2(0,1;V)\rightarrow\R$ and the mapping $\Gamma$ are of class $\mathcal{C}^1$, and their respective differentials at some point $(q,v)$ are given by
\begin{equation*}
\begin{split}
dJ_{(q,v)}. (\delta q,\delta v) = &
\int_0^1 \left( \langle\partial_qL(q(t),v(t)),\delta q(t)\rangle_{X^*,X}+\langle\partial_vL(q(t),v(t)),\delta v(t)\rangle_{V^*,V} \right) dt \\
&\quad 
+ \langle dg_{q(1)},\delta q(1)\rangle_{X^*,X}  ,
\end{split}
\end{equation*}
for all $(\delta q,\delta v)\in H_{0}^1(0,1;X)\times L^2(0,1;V)$, and $d\Gamma_{(q,v)} = (d{\Gamma_1}_{(q,v)}, d{\Gamma_2}_{(q,v)})$ with
\begin{equation*}
\begin{split}
(d{\Gamma_1}_{(q,v)}.(\delta q,\delta v)) (t) &=
\dot{\delta q}(t)-\partial_q\xi_{q(t)}v(t). \delta q(t)-\xi_{q(t)}\delta v(t) ,\\
(d{\Gamma_2}_{(q,v)}.(\delta q,\delta v)) (t) &= \partial_qC_{q(t)}v(t). \delta q(t)+C_{q(t)}\delta v(t) ,
\end{split}
\end{equation*}
for almost every $t\in[0,1]$.

\begin{lemma}\label{ift}
For every $(q,v)\in H_{q_0}^1(0,1;M)\times L^2(0,1;V)$, the linear continuous mapping
$
d\Gamma_{(q,v)}:H^1_{0}(0,1;X)\times L^2(0,1;V)\rightarrow L^2(0,1;X)\times L^2(0,1;Y)
$ 
is surjective. Moreover the mapping $\partial_q{\Gamma_1}_{(q,v)}:H^1_{0}(0,1;M)\rightarrow L^2(0,1;X)$ is an isomorphism.
\end{lemma}

\begin{proof}
Let $(q,v)\in H_{q_0}^1(0,1;M)\times L^2(0,1;V)$ and $(a,b)\in L^2(0,1;X)\times L^2(0,1;Y)$. 
Let us prove that there exists $(\delta q,\delta v)\in H^1_{0}(0,1;X)\times L^2(0,1;V)$ such that
\begin{eqnarray}
a(t) &=& \dot{\delta q}(t)-\partial_q\xi_{q(t)}v(t). \delta q(t)-\xi_{q(t)}\delta v(t) ,\label{eeq1} \\
b(t) &=& \partial_qC_{q(t)}v(t). \delta q(t)+C_{q(t)}\delta v(t) ,\label{eeq2}
\end{eqnarray}
for almost every $t\in [0,1]$.

For every $\tilde q\in M$, $C_{\tilde q}:V\rightarrow Y$ is a surjective linear continuous mapping. It follows that $C_{\tilde q\vert (\ker (C_{\tilde q}))^{\perp}}:(\ker (C_{\tilde q}))^{\perp}\rightarrow Y$ is an isomorphism (note that $V=\ker (C_{tilde q})\oplus (\ker (C_{\tilde q}))^{\perp}$ since $V$ is Hilbert). We set $A_{\tilde q}=(C_{\tilde q\vert (\ker (C_{\tilde q}))^{\perp}})^{-1}=K_VC_{\tilde q}^*(C_{\tilde q}K_VC_{\tilde q}^*)^{-1}$. Note that $\tilde q\mapsto A_{\tilde q}$ is of class $\mathcal{C}^1$ in a neighbourhood of $q([0,1])$.

Assume for the moment that $\delta q(\cdot)$ is known. Then we choose $\delta v(\cdot)$ defined by
$\delta v(t)=A_{q(t)}\left( b(t)-\partial_qC_{q(t)}. \delta q(t) \right)$ for almost every $t\in [0,1]$, so that \eqref{eeq2} is satisfied. Plugging this expression into \eqref{eeq1} yields
$$
\dot{\delta q}(t)-\partial_q\xi_{q(t)}v(t).\delta q(t)-\xi_{q(t)}A_{q(t)}\left( b(t)-\partial_qC_{q(t)}v(t).\delta q(t) \right) = a(t),
$$
for almost every $t\in[0,1]$. This is a well-posed linear differential equation with square-integrable coefficients in the Banach space $X$, which has a unique solution $\delta q\in H^1_0(0,1;X)$ such that $\delta q(0)=0$. 
This proves the statement.

Proving that the mapping $\partial_q{\Gamma_1}_{(q,v)}:H^1_{0}(0,1;M)\rightarrow L^2(0,1;X)$, defined by $(\partial_q{\Gamma_1}_{(q,v)}.\delta q)(t) = \dot{\delta q}(t)-\partial_q\xi_{q(t)}v(t).\delta q(t)$ for almost every $t\in[0,1]$, is an isomorphism follows the same argument, by Cauchy uniqueness.
\end{proof}

Let $(q,v)\in H_{q_0}^1(0,1;M)\times L^2(0,1;V)$ be an optimal solution of the optimal control problem. In other words, $(q,v)$ is a minimizer of the problem of minimizing the functional $J$ over the set of constraints $\Gamma^{-1}(\{ 0\})$ (which is a $\mathcal{C^1}$ manifold as a consequence of Lemma \ref{ift} and of the implicit function theorem).
Since $d\Gamma_{(q,v)}$ is surjective (note that this fact is essential since we are in infinite dimension), it follows from \cite[Theorem 4.1]{Kurcyusz} that there exists a nontrivial Lagrange multiplier
$ (p,\lambda)\in L^2(0,1;X)^*\times L^2(0,1;Y)^*$ such that
$dJ_{(q,v)}+(d\Gamma_{(q,v)})^*(p,\lambda)=0$.
Moreover since $Y$ is reflexive one has $L^2(0,1;Y)^*=L^2([0,1],Y^*)$, and hence we can identify $\lambda$ with a square-integrable $Y^*$-valued measurable mapping, so that the Lagrange multipliers relation yields
\begin{equation}\label{djdg}
\begin{split}
0 & = \big\langle dJ_{(q,v)}+(d\Gamma_{(q,v)})^*(p,\lambda), (\delta q,\delta v)\big\rangle \\
&= 
\big\langle p,\dot{\delta q}\big\rangle_{L^2(0,1;X)^*,L^2(0,1;X)} - \big\langle p, \partial_q\xi_{q}v.\delta q\big\rangle_{L^2(0,1;X)^*,L^2(0,1;X)}- \big\langle p, \xi_{q}\delta v\big\rangle_{L^2(0,1;X)^*,L^2(0,1;X)}
\\
&\quad + \int_0^1\left\langle \partial_qL(q(t),v(t)) + (\partial_qC_{q(t)}v(t))^*
\lambda(t) , \delta q(t)\right\rangle_{X^*,X} dt \\
&\quad + \int_0^1 \left\langle \partial_vL(q(t),v(t)) + C_{q(t)}^* \lambda(t) , \delta v(t)\right\rangle_{V^*,V} dt 
 + \big\langle dg_{q(1)}, \delta q(1) \big\rangle_{X^*,X} ,
\end{split}
\end{equation}
for all $(p,\lambda)\in L^2(0,1;X)^*\times L^2([0,1],Y^*)$ and all $(\delta q,\delta v)\in H^1_0(0,1;M)\times L^2(0,1;V)$.
Note that the space $L^2(0,1;X)^*$ can be different from $L^2(0,1;X^*)$ (unless $X^*$ satisfies the Radon-Nikodym property, but there is no reason to consider such a Banach space $X$), and hence \textit{a priori} $p$ cannot be obviously identified with a square-integrable $X^*$-valued measurable mapping. Anyway, in the next lemma we show that this identification is possible, due to a hidden regularity property in \eqref{djdg}.

\begin{lemma}\label{lem4}
We can identify $p$ with an element of $L^2([0,1],X^*)$, so that
$$
\big\langle p, r \big\rangle_{L^2(0,1;X^*),L^2(0,1;X)} = \int_0^1 \big\langle p(t), r(t) \big\rangle_{X^*,X} \, dt,
$$
for every $r\in L^2(0,1;X)$.
\end{lemma}

\begin{proof}[Proof of Lemma \ref{lem4}]
For every $r\in L^2(0,1;X)$, we define $\delta q\in H^1_0(0,1;X)$ by $\delta q(s)=\int_0^sr(t)\, dt$ for every $s\in [0,1]$ (Bochner integral in the Banach space $X$), so that $r=\dot{\delta q}$.
Defining $\alpha\in L^2(0,1;X)^*$ by
\begin{equation}\label{alpha}
\begin{split}
\big\langle \alpha, f \big\rangle_{L^2(0,1;X)^*, L^2(0,1;X)} = &
\ \big\langle p , \partial_q\xi_{q}v.f \big\rangle_{L^2(0,1;X)^*, L^2(0,1;X)} \\
& \ - \int_0^1\left\langle \partial_qL(q(t),v(t))+(\partial_qC_{q(t)}v(t))^* \lambda(t) , f(t)\right\rangle_{X^*,X} dt ,
\end{split}
\end{equation}
for every $f\in L^2(0,1;X)$, and taking $\delta v=0$ in \eqref{djdg}, we get
\begin{equation}\label{train16h06}
\big\langle p, r\big\rangle_{L^2(0,1;X)^*, L^2(0,1;X)} = 
\big\langle \alpha, \delta q\big\rangle_{L^2(0,1;X)^*, L^2(0,1;X)} - \big\langle dg_{q(1)}, \delta q(1) \big\rangle_{X^*,X} .
\end{equation}
Let us express $\delta q$ in another way with respect to $r$. By definition, one has $\delta q(s) = \int_0^s r(t)\, dt$ for every $s\in [0,1]$, and this can be also written as $\delta q(s) = \int_0^1 \chi_{[t,1]}(s) r(t)\, dt$, with $\chi_{[t,1]}(s)=1$ whenever $s\in [t,1]$ and $0$ otherwise. In other words, one has $\delta q = \int_0^1 \chi_{[t,1]} r(t)\, dt$ (Bochner integral). For every $t\in[0,1]$, we define the operator $A_t:X\rightarrow L^2(0,1;X)$ by $A_tx=\chi_{[t,1]}x$. It is clearly linear and continuous. Then, we have $\delta q = \int_0^1 A_t(r(t))\, dt$, and therefore, using \eqref{train16h06},
\begin{equation*}
\big\langle p, r\big\rangle_{L^2(0,1;X)^*, L^2(0,1;X)} = 
\Big\langle \alpha, \int_0^1 A_t(r(t))\, dt\Big\rangle_{L^2(0,1;X)^*, L^2(0,1;X)} - \big\langle dg_{q(1)}, \int_0^1 r(t)\, dt \big\rangle_{X^*,X} .
\end{equation*}
Now, interchanging the Bochner integrals and the linear forms, we infer that
$$
\big\langle p, r\big\rangle_{L^2(0,1;X)^*, L^2(0,1;X)} =
\int_0^1 \big\langle \alpha, A_t(r(t))\big\rangle_{L^2(0,1;X)^*, L^2(0,1;X)} \, dt
- \int_0^1 \big\langle dg_{q(1)}, r(t)\big\rangle_{X^*,X}\, dt ,
$$
and then, using the adjoint $A_t^*: L^2(0,1;X)^*\rightarrow X^*$, we get
$$
\big\langle p, r\big\rangle_{L^2(0,1;X)^*, L^2(0,1;X)} =
\int_0^1 \big\langle A_t^*\alpha - dg_{q(1)}, r(t)\big\rangle_{X^*,X} \, dt .
$$
Since this identity holds true for every $r\in L^2(0,1;X)$, it follows that $p$ can be identified with an element of $L^2(0,1;X^*)$, still denoted by $p$, with $p(t) = A_t^*\alpha - dg_{q(1)}$ for almost every $t\in[0,1]$.
\end{proof}

Still using the notation introduced in the proof of Lemma \ref{lem4}, now that we know that $p\in L^2(0,1;X^*)$, we infer from \eqref{alpha} that $\alpha$ can as well be identified with an element of $L^2([0,1],X^*)$, with
$$
\alpha(t)=(\partial_q\xi_{q(t)}v(t))^*p(t) - \partial_qL(q(t),v(t))-(\partial_qC_{q(t)}v(t))^*
\lambda(t),
$$ 
for almost every $t\in[0,1]$. Note that $\alpha(t)=\partial_qH(q(t),p(t),v(t),\lambda(t))$, where the Hamiltonian $H$ is defined by \eqref{defHam}.

Since $\alpha\in L^2(0,1;X^*)$, we have, for every $x\in X$, 
\begin{equation*}
\begin{split}
\langle A_t^*\alpha,x\rangle_{X^*,X} 
&= \big\langle\alpha, \chi_{[t,1]}x\big\rangle_{L^2(0,1,X^*),L^2(0,1;X)} 
= \int_0^1 \langle\alpha(s), \chi_{[t,1]}(s)x\rangle_{X^*,X} \, ds \\
&= \int_t^1 \langle\alpha(s), x\rangle_{X^*,X}\, ds 
= \Big\langle\int_t^1\alpha(s)\, ds, x\Big\rangle_{X^*,X} ,
\end{split}
\end{equation*}
and therefore $A_t^*\alpha=\int_t^1\alpha(s)\, ds$ for every $t\in[0,1]$.
It follows that $p(t) = A_t^*\alpha - dg_{q(1)} = \int_t^1\alpha(s)\, ds  - dg_{q(1)}$, and hence, that $p\in H^1(0,1;X^*)$ and that $p(\cdot)$ satisfies the differential equation $\dot{p}(t)=-\alpha(t)=-\partial_qH(q(t),p(t),v(t),\lambda(t))$ for almost every $t\in[0,1]$ and $p(1)+dg_{q(1)}=0$. 

Finally, taking $\delta q=0$ in \eqref{djdg} yields
$$
\int_0^1 \big\langle \partial_vL(q(t),v(t))+C_{q(t)}^*\lambda(t)-\xi_{q(t)}^*p(t) , \delta v(t)\big\rangle)_{V^*,V}\, dt=0,
$$
for every $\delta v\in L^2(0,1;X)$, and hence $\partial_vL(q(t),v(t))+C_{q(t)}^* \lambda(t)-\xi_{q(t)}^*p(t)=0$ for almost every $t\in[0,1]$, which exactly means that $\partial_vH(q(t),p(t),v(t),\lambda(t))=0$.
The theorem is proved.
\end{proof}

\subsection{The geodesic equations in a shape space}\label{geodeqshsp}
We use the notations introduced in Section \ref{modshsp}, and consider a shape space $M$ of order $\ell\in\N,$ $V$ an RKHS  of vector fields of class $\mathcal{C}^{\ell+1}_0$ on $\R^d$, and we set $L(q,v)=\frac{1}{2}\Vert v\Vert_V^2$. We assume that $C$ and $g$ are at least of class $\mathcal{C}^1$ and that $C_q:V\rightarrow Y$ is surjective for every $q\in M$. 

In this context $\xi$ is of class $\mathcal{C}^1$. Let us apply Theorem \ref{PMP}.
We have $\partial_vH(q,p,v,\lambda)=\xi_{q}^*p-C_{q}^*\lambda-( v,\cdot)_V$, and the condition $\partial_vH(q,p,v,\lambda)=0$ is equivalent to $v=K_V(\xi_{q}^*p-C_{q}^*\lambda)$.
Then we have $\partial_pH=\xi_qv=\xi_{q}K_V(\xi_{q}^*p-C_{q}^*\lambda)=K_{q}p-\xi_{q}K_VC_{q}^*\lambda$, where $K_q=\xi_{q}K_V\xi_{q}^*$ is defined by \eqref{Kq}.
Besides, $C_qv=C_qK_V\xi_{q}^*p-C_qK_VC_{q}^*\lambda=0$ if and only if
$C_qK_VC_q^*\lambda=C_qK_V\xi_{q}^*p$.
Since $C_q$ is surjective, it follows that $C_qK_VC_q^*$ is invertible, and hence
$\lambda=\lambda_{q,p}=(C_qK_VC_q^*)^{-1}C_qK_V\xi_{q}^*p$.
The mapping $(q,p)\mapsto \lambda_{q,p}$ defined as such is of class $\mathcal{C}^1$ and is linear in $p$. In particular, $v=v_{q,p}=K_V(\xi_{q}^*p-C_{q}^*\lambda_{q,p})$ is a function of class $\mathcal{C}^1$ of $q$ and $p$ and is linear in $p$. We have obtained the following result.

\begin{theorem}[Geodesic equations in shape spaces]\label{geodeq}
Let $(q(\cdot),v(\cdot))\in H^1_{q_0}(0,1;M)\times L^2(0,1;V)$ be a solution of Problem \ref{diffprob1}. There exists $p\in H^1(0,1;X^*)$ such that
$$
v(t)=v_{q(t),p(t)}=K_V\left(\xi_{q(t)}^*p(t)-C_{q(t)}^*\lambda(t) \right),
$$
for almost every $t\in[0,1]$, and $p(\cdot)$ satisfies $p(1)+dg_{q(1)}=0$ and the geodesic equations
\begin{equation}\label{geodeq1}
\begin{aligned}
\dot{q}(t)&=K_{q(t)}p(t)-\xi_{q(t)}K_VC_{q(t)}^*\lambda(t),\\ 
\dot{p}(t)&=-\partial_q\left\langle p(t),\xi_{q(t)}v(t)\right\rangle_{X^*,X} + \partial_q\left\langle\lambda(t), C_{q(t)}v(t)\right\rangle_{Y^*,Y},
\end{aligned}
\end{equation}
for almost every $t\in[0,1]$, with
$$
\lambda(t)=\lambda_{q(t),p(t)}=(C_{q(t)}K_VC_{q(t)}^*)^{-1}C_{q(t)}K_V\xi_{q(t)}^*p(t).
$$
Moreover the mapping $t\mapsto \frac{1}{2}\Vert v(t)\Vert^2$ is constant, and one has
$$
J(v)=\frac{1}{2}\Vert v(0)\Vert^2_V+g(q(1)).
$$
\end{theorem}

\begin{remark}\label{rem11}
Defining the so-called reduced Hamiltonian $h:M\times X^*\rightarrow \R$ by 
$$h(q,p)=H(q,p,v_{q,p},\lambda_{q,p}),$$
we have a priori $\partial_qh=\partial_qH+\partial_vH(q,p,v_{q,p},\lambda_{q,p}) \circ\partial_q(v_{q,p})+\partial_\lambda H(q,p,v_{q,p},\lambda_{q,p})\circ \partial_q(\lambda_{q,p})$. But since $\lambda_{q,p}$ and $v_{q,p}$ are such that $\partial_vH(q,p,v_{q,p},\lambda_{q,p})=0$ and $C_qv_{q,p}=\partial_\lambda H(q,p,v_{q,p},\lambda_{q,p})=0$, it follows that $\partial_qh(q,p)=\partial_qH(q,p,v_{q,p},\lambda_{q,p})$. Similarly, we have $\partial_ph(q,p)=\partial_pH(q,p,v_{q,p},\lambda_{q,p})$.
Therefore, in Theorem \ref{geodeq}, the geodesics are the solutions of the Hamiltonian system
$$
\dot{q}(t)=\partial_ph(q(t),p(t)),\quad  \dot{p}(t)=-\partial_qh(q(t),p(t)). 
$$
\end{remark}

\begin{corollary}\label{geodeq3}
Assume that the mappings $C$ and $\xi$ are of class $\mathcal{C}^2$. Then $h$ is of class $\mathcal{C}^2$ as well, and for every $(q_0,p_0)\in M\times X^*$, there exists $\varepsilon>0$ and there exists a unique solution $(q,p):[0,\epsilon]\rightarrow M\times X^*$ of the geodesic equations \eqref{geodeq1} such that $(q(0),p(0))=(q_0,p_0)$.
\end{corollary}

Note that for most of shape spaces (at least, for all shape spaces given as examples in this paper), the mapping $\xi$ is of class $\mathcal{C}^2$ whenever $V$ is an RKHS of vector fields of class $\mathcal{C}^{\ell+2}$.

\begin{example}
A geodesic $(q(t),p(t))=(x_1(t),\dots,x_n(t),p_1(t),\dots,p_n(t))$ on the landmark space $\mathrm{Lmk}_d(n)$ must satisfy the equations
$$
\dot{x}_i(t)=\sum_{j=1}^nK(x_i(t),x_j(t))p_j(t),
\quad
\dot{p}_j(t)=-\frac{1}{2}\sum_{j=1}^n\partial_{x_i}\left(p_i(t)^TK(x_i(t),x_j(t))p_j(t)\right),
$$
where $K$ is the kernel of $V$.
\end{example}

\begin{remark}\label{rem_propequiv}
Assume that the constraints are kinetic, i.e., are of the form $C_q\xi_qv=0$ with $C_q:X\rightarrow Y$. Then $\dot{q}(t)=K_{q(t)}(p(t)-C_{q(t)}^*\lambda(t))$ and $\lambda_{q,p}=(C_qK_qC_q^*)^{-1}C_qK_qp$.
Note that even if $K_q$ is not invertible, the assumption that $C_q\xi_q$ is surjective means that $C_qK_qC_q^*$ is invertible. 
%Now, denoting $\pi_qp=p-C_{q}^*\lambda(q,p)$, by carefully gathering terms together we get
%$$
%\partial_qH (q,p,v_{q,p},\lambda_{q,p})=\frac{1}{2}\langle\pi_qp, \partial_qK_q\pi_qp \rangle_{X^*,X}-\langle\lambda_{q,p}, \partial_qC_{q}.K_q\pi_qp\rangle_{Y^*,Y},
%$$
%and hence the geodesic equations \eqref{geodeq1} in shape spaces can be rewritten as
%\begin{equation}\label{geodeq2}
%\begin{split}
%\dot{q}(t)&=K_{q(t)}\pi_{q(t)}p(t),\\
%\dot{p}(t)&=-\frac{1}{2}\langle\pi_{q(t)}p(t), \partial_qK_{q(t)}\pi_{q(t)}p(t) \rangle_{X^*,X}+\langle\lambda_{q(t),p(t)}, \partial_qC_{q(t)}.K_{q(t)}\pi_{q(t)}p(t)\rangle_{Y^*,Y}.
%\end{split}
%\end{equation}
%with $\lambda_{q,p}=(C_qK_qC_q^*)^{-1}C_qK_qp$.
%Note that $\pi_qp$ is the orthogonal projection of $p$ on $\ker (C_q)$ for the positive semi-definite bilinear form on $X^*$ induced by $K_q$. 
In particular, minimizing controls take the form $K_V\xi_q^*u$, with $u\in L^2(0,1;X^*)$. This proves the contents of Proposition \ref{equiv} (see Section \ref{modshsp}).
\end{remark}

\begin{remark}\label{rem_abnormal}
Let $M$ (resp., $M'$), open subset of a Banach space $X$ (resp., $X'$), be a shape space of order $\ell$ (resp. of order $\ell'\leq\ell$). Assume that there is a dense and continuous inclusion $X\hookrightarrow X'$, such that $M\hookrightarrow M'$ is equivariant (in particular we have $X'^*\hookrightarrow X^*$).
For every $q_0\in M\subset M'$, there are more geodesics emanating from $q_0$ on $M$ than on $M'$ (indeed it suffices to consider initial momenta $p_0\in X^*\setminus X'^*$).
These curves are not solutions of the geodesic equations on $M'$, and are an example of so-called \textit{abnormal extremals} \cite{AgrachevSachkov,TBOOK}. Note that they are however not solutions of Problem \ref{diffprob1} specified on $X'$, since Theorem \ref{PMP} implies that such solutions are projections of geodesics having an initial momentum in $X'^*$.

An example where this situation is encountered is the following. Let $S$ be a compact Riemannian manifold. Consider $X=\mathcal{C}^{\ell+1}(S,\R^d)$ and $X'=\mathcal{C}^{\ell}(S,\R^d)$, with actions defined in Definition \ref{infgaction}. If $p(0)$ is a $\ell+1$-th order distribution, then the geodesic equations on $X$ with initial momentum $p(0)$ yield an abnormal geodesic in $X'$.
\end{remark}

\begin{remark}
Following Remark \ref{rem_abnormal}, if the data attachment function $g:X\rightarrow \R^+$ can be smoothly extended to a larger space $X'$, then the initial momentum of any solution of Problem \ref{diffprob1} is actually in $X'^*\subset X^*$.

For example, we set $g(q)=\int_{S}\vert q(s)-q_{target}(s)\vert^2 \, ds$, for some $q_{target}\in X$, with $X=\mathcal{C}^{\ell}(S,\R^d)$ and $S$ a compact Riemannian manifold. Let $X'=L^2(S,\R^d)$. Then the momentum $p:[0,1]\rightarrow X^*$ associated with a solution of Problem \ref{diffprob1} (whose existence follows from Theorem \ref{geodeq}) takes its values in $X'^*=L^2(S,\R^d)$.
\end{remark}

\paragraph{The case of pure state constraints.}
Let us consider pure state constraints, i.e., constraints of the form $C(q)=0$ with $C:M\rightarrow Y$. Recall that, if $C$ is of class $\mathcal{C}^1$, then they can be transformed into the mixed constraints $dC_q.\dot{q}=dC_q.\xi_qv=0$. In this particular case, the geodesic equations take a slightly different form.

\begin{proposition}\label{geodstco}
Assume that $C:M\rightarrow Y$ is of class $\mathcal{C}^3$, and that $dC_q.\xi_q:V\rightarrow Y$ is surjective for every $q\in M$. 
Consider Problem \ref{diffprob1} with the pure state constraints $C(q)=0$.
If $v(\cdot)\in L^2(0,1;V)$ is a solution of Problem \ref{diffprob1}, associated with the curve $q(\cdot):[0,1]\rightarrow M$, then there exists $\tilde{p}:[0,1]\rightarrow X^*$ such that
$$
\dot{q}(t)=K_{q(t)}\tilde{p}(t),\quad
\dot{\tilde{p}}(t)=-\frac{1}{2}\partial_q\langle\tilde{p}(t), K_{q(t)}\tilde{p}(t)\rangle_{X^*,X} + \partial_q\langle\tilde{\lambda}_{q(t),\tilde{p}(t)}, C(q(t))\rangle_{Y^*,Y},
$$
for almost every $t\in[0,1]$, with
$$
\tilde{\lambda}_{q,\tilde{p}}=(dC_qK_qdC_q^*)^{-1}\left(\frac{1}{2}dC_q. K_q\partial_q\langle\tilde{p}, K_q\tilde{p}\rangle_{X^*,X}- d^2C_q.(K_q\tilde{p},K_q\tilde{p}) - dC_q.\left(\partial_q(K_q\tilde{p}). K_q\tilde{p}\right)\right).
$$
Moreover one has $dC_{q_0}.K_{q_0}\tilde{p}_0=0$ and $dg_{q(1)}+\tilde{p}(1)=dC_{q(1)}^*\nu$ for some $\nu\in Y^*$. 
\end{proposition}

Note that all functions involved are of class $\mathcal{C}^1$. Therefore, for a given initial condition $(q_0,\tilde{p}_0)\in T_{q_0}^*M$ with $dC_{q_0}.K_{q_0}p_0=0$, there exists a unique geodesic emanating from $q_0$ with initial momentum $\tilde{p}_0$.

\begin{proof}
The proof just consists in considering $\tilde{\lambda}(t)=\dot{\lambda}(t)$ and $\tilde{p}=p-dC^*_q.\lambda$, where $p$ and $\lambda$ are given by Theorem \ref{geodeq}, and then in differentiating $C(q(t))=0$ twice with respect to time.
\end{proof}

\begin{remark}
The mappings $\tilde{p}$ and $\tilde{\lambda}$ can also be obtained independently of Theorem \ref{geodeq} as Lagrange multipliers for the mapping of constraints $\Gamma(q,v)=(\dot{q}-\xi_{q}v,C(q))$, using the same method as in the proof of Theorem \ref{PMP}, under the assumptions that $C:X\rightarrow Y$ is of class $\mathcal{C}^2$ and $dC_q.\xi_q$ is surjective for every $q\in M$.
\end{remark}

\section{Algorithmic procedures}\label{sec_algos}
In this section we derive some algorithms in order to compute the solutions of the optimal control problem considered throughout. We first consider problems without constraint in Section \ref{unconsmin}, and then with constraints in Section \ref{matchconst}.

\subsection{Problems without constraints}\label{unconsmin}
Shape deformation analysis problems without constraint have already been studied in \cite{BMTY,DGM,GTY1,GTY2,JM,MTY1,MTY2,T1} with slightly different methods, in different and specific contexts. With our previously developed general framework, we are now going to recover methods that are well known in numerical shape analysis, but with a more general point of view allowing us to generalize the existing approaches.

\paragraph{Gradient Descent.}
We adopt the notations, the framework and the assumptions used in Section \ref{PMPs}, but without constraint. For $q_0\in M$ fixed, we consider the optimal control problem of minimizing the functional $J$ defined by \eqref{defJqv} over all $(q(\cdot),v(\cdot))\in H^1_{q_0}(0,1;M)\times L^2(0,1;V)$ such that $\dot{q}(t)=\xi_{q(t)}v(t)$ for almost every $t\in[0,1]$.
The Hamiltonian of the problem then does not involve the variable $\lambda$, and is the function $H:M\times X^*\times V\rightarrow \R$ defined by $H(q,p,v)=\langle p,\xi_qv\rangle_{X^*,X}-L(q,v)$.

We assume throughout that $g\in\mathcal{C}^1(X,\R)$, that $\xi:M\times V\rightarrow X$ is of class $\mathcal{C}^{2}$, is a linear mapping in $v$, and that $L\in\mathcal{C}^{2}(X\times V,\R)$ satisfies the estimate \eqref{estilag}.

As in the proof of Theorem \ref{PMP}, we define the mapping $\Gamma:H_{q_0}^1(0,1;M)\times L^2(0,1;V)\rightarrow L^2(0,1;X)$ by $\Gamma(q,v)(t) = \dot{q}(t)-\xi_{q(t)}v(t)$ for almost every $t\in[0,1]$. 
The objective is to minimize the functional $J$ over the set $E=\Gamma^{-1}(\{0\})$.

According to Lemma \ref{ift}, the mapping $\partial_q\Gamma_{(q,v)}$ is an isomorphism for all $(q,v)\in E$. Therefore, the implicit function theorem implies that $E$ is the graph of the mapping $v\mapsto q_v$ which to a control $v\in L^2(0,1;V)$ associates the curve $q_v\in H^1_{q_0}(0,1;X)$ solution of $\dot{q}_v(t)=\xi_{q_v(t)}v(t)$ for almost every $t\in[0,1]$ and $q_v(0)=q_0$. Moreover this mapping is, like $\Gamma$, of class $\mathcal{C}^2$.

Then, as it was already explained in Remark \ref{rempb1_unknowns} in the case where $L(q,v)=\frac{1}{2}\Vert v\Vert_V^2$, minimizing $J$ over $E$ is then equivalent to minimizing the functional $J_1(v)=J(q_v,v)$ over $L^2(0,1;V)$. 

Thanks to these preliminary remarks, the computation of the gradient of $J_{\vert E}$ then provides in turn a gradient descent algorithm.

\begin{proposition}\label{dcost}
The differential of $J_1$ is given by
$$
d{J_1}_v. \delta v = -\int_0^1 \partial_v H(q_v(t),p(t),v(t)). \delta v(t) \, dt,
$$
for every $\delta v\in L^2(0,1;V)$, where $p\in H^1([0,1],X^*)$ is the solution of $\dot{p}(t)=-\partial_qH(q_v(t),p(t),v(t))$ for almost every $t\in [0,1]$ and $p(1)+dg_{q_v(1)}=0$.
In particular we have 
$$
\nabla J_1(v)(t)=-K_V\partial_v H(q_v(t),p(t),v(t)),
$$
for almost every $t\in[0,1]$.
\end{proposition}

\begin{remark}
This result still holds true for Lagrangians $L$ that do not satisfy \eqref{estilag}, replacing $L^2(0,1;V)$ with $L^\infty(0,1;V)$. The gradient is computed with respect to the pre-Hilbert scalar product inherited from $L^2$.
\end{remark}

\begin{proof}
Let $v\in L^2(0,1;V)$ be arbitrary. For every $\delta v\in L^2(0,1;V)$, we have $d{J_1}_v.\delta v = dJ_{(q_v,v)}. (\delta q,\delta v)$, with $\delta q = dq_v.\delta v$. Note that $(\delta q,\delta v)\in  T_{(q_v,v)}E$ (the tangent space of the manifold $E$ at $(q_v,v)$), since $E$ is the graph of the mapping $v\rightarrow q_v$.
Since $E=\Gamma^{-1}(\{0\})$, we have $\langle d\Gamma_{(q_v,v)}^*p, (\delta q,\delta v)\rangle=0$ for every $p\in L^2(0,1;X)^*$. Let us find some particular $p$ such that
$\langle d{J_1}_{(q_v,v)}+d\Gamma_{(q_v,v)}^*p, (\delta q,\delta v)\rangle$ only depends on $\delta v$.

Let $p\in H^1([0,1],X^*)$ be the solution of $\dot{p}(t)=-\partial_qH(q_v(t),p(t),v(t))$ for almost every $t\in [0,1]$ and $p(1)+dg_{q_v(1)}=0$.
Using the computations done in the proof of Theorem \ref{PMP}, we get
\begin{equation*}
\begin{split}
\langle d{J_1}_{(q_v,v)}+d\Gamma_{(q_v,v)}^*p , (\delta q,\delta v)\rangle
 = 
\int_0^1\Big( 
& \langle p(t),\dot{\delta q}(t)\rangle_{X^*,X}-\langle\partial_qH(q_v(t),p(t),v(t)),\delta q(t)\rangle_{X^*,X} \\
& - \langle\partial_vH(q_v(t),p(t),v(t)), \delta v(t)\rangle_{V^*,V} \Big) \, dt + dg_{q_v(1)}. \delta q_v(1).
\end{split}
\end{equation*}
Integrating by parts and using the relations $\delta q(0)=0$ and $p(1)+dg_{q_v(1)}=0$, we obtain
$$
\langle d{J_1}_{(q_v,v)}+d\Gamma_{(q_v,v)}^*p , (\delta q,\delta v)\rangle
=
-\int_0^1 \partial_v H(q_v(t),p(t),v(t)). \delta v(t) \, dt
$$
Since $\langle d\Gamma_{(q_v,v)}^*p, (\delta q,\delta v)\rangle=0$, the proposition follows.
\end{proof}

In the case of shape spaces, which form our main interest here, we have $L(q,v)=\frac{1}{2}\Vert v\Vert_V^2$, and then $\partial_vH(q,p,v)=\xi_q^*p-(v,\cdot)_V$ and $K_V\partial_vH(q,p,v)=K_V\xi_q^*p-v$.
It follows that
$$\nabla J_1(v)=v-K_V\xi_{q_v}^*p.$$
In particular, if $v=K_V\xi_q^*u$ for some $u\in L^2([0,1],X^*)$, then
$$
v-h\nabla J_1(v) = K_v\xi_q^*(u-h(u-p)),
$$
for every $h\in \R$. Therefore, applying a gradient descent algorithm does not change this form.
It is then important to notice that this provides as well a gradient descent algorithm for solving Problem \ref{diffprob2} (the kernel formulation of Problem \ref{diffprob1}) without constraint, and this in spite of the fact that $L^2([0,1],X^*)$ is not necessarily a Hilbert space. 
In this case, if $p\in H^1([0,1],X^*)$ is the solution of $\dot{p}(t)=-\partial_qH(q_v(t),p(t),v(t))$ for almost every $t\in [0,1]$ and $p(1)+dg_{q_v(1)}=0$, then $u-p$ is the gradient of the functional $J_2$ defined by \eqref{defJ2} with respect to the symmetric nonnegative bilinear form $B_q(u_1,u_2)=\int_0^1 \langle u_1(t), K_{q(t)}u_2(t)\rangle_{X^*,X}\,dt$.

This gives a first algorithm to compute unconstrained minimizers in a shape space. We next provide a second method using the space of geodesics.

\paragraph{Gradient descent on geodesics: minimization through shooting.}
Since the tools are quite technical, in order to simplify the exposition we assume that the shape space $M$ is finite dimensional, i.e., that $X=\R^n$ for some $n\in\N$. The dual bracket $\langle p,w\rangle_{X^*,X}$ is then identified with the canonical Euclidean product $p^Tw$, and $K_q=\xi_qK_V\xi_q^*:X^*\rightarrow X$ is identified with a $n\times n$ positive semi-definite symmetric matrix.
Theorem \ref{geodeq} and Corollary \ref{geodeq3} (see Section \ref{geodeqshsp}) imply that the minimizers of the functional $J_1$ defined by \eqref{defJ1} coincide with those of the functional
\begin{equation}
\hat{J}_1(p_0)=\frac{1}{2}p_0^TK_{q_0}p_0+g(q(1)),
\label{reducedcost}
\end{equation}
where $v=K_V\xi_q^*p$ and $(q(\cdot),p(\cdot))$ is the geodesic solution of the Hamiltonian system $\dot{q}(t)=\partial_p h(q(t),p(t))$, $\dot{p}(t)=-\partial_q h(q(t),p(t))$, for almost every $t\in[0,1]$, with $(q(0),p(0))=(q_0,p_0)$. Here, $h$ is the reduced Hamiltonian (see Remark \ref{rem11}) and is given by
$h(q,p) = \frac{1}{2}\langle p, \xi_qK_V\xi_q^*p\rangle_{X^*,X}= \frac{1}{2}\langle p, K_qp\rangle_{X^*,X}$.
Therefore, computing a gradient of $\hat{J}_1$ for some appropriate bilinear symmetric nonnegative product provides in turn another algorithm for minimizing the functional $J_1$. For example, if the inner product that we consider is the canonical one, then $\nabla \hat{J}_1(p_0)=K_{q_0}p_0+\nabla (g\circ q(1))(p_0)$.
The term $\nabla(g\circ q(1))(p_0)$ is computed thanks to the following well-known result.

\begin{lemma}\label{gradedo}
Let $\ n\in\N$, let $U$ be an open subset of $\R^n$, let $f:U\rightarrow\R^n$ be a complete smooth vector field on $U$, let $G$ be the function of class $\mathcal{C}^1$ defined on $U$ by
$G(q_0)=g(q(1))$, where $g$ is a function on $U$ of class $\mathcal{C}^1$ and $q:[0,1]\rightarrow\R^n$ is the solution of $\dot{q}(t)=f(q(t))$ for almost every $t\in[0,1]$ and $q(0)=q_0$. Then $\nabla G(q_0)=Z(1)$ where $Z:[0,1]\rightarrow\R^n$ is the solution of $\dot{Z}(t)=df_{q(1-t)}^TZ(t)$ for almost every $t\in[0,1]$ and $Z(0)=\nabla g(q(1))$.
\end{lemma}

In our case, we have $U=M\times\R^n$ and $f(q,p) = (\nabla_ph, -\nabla_qh) = (K_qp, -\frac{1}{2}\nabla_q(p^TK_qp))$.
Note that we used the Euclidean gradient instead of the derivatives. This is still true thanks to the identification made between linear forms and vectors at the beginning of the section.
We get $ \nabla\hat{J}_1(p_0)=K_{q_0}p_0+\alpha(1)$, where $Z(\cdot)= (z(\cdot), \alpha(\cdot))$ is the solution of $\dot{Z}(t)=df_{q(1-t),p(1-t)}^TZ(t)$ for almost every $t\in[0,1]$ and $Z(0) = (z(0),\alpha(0)) = (\nabla g(q(1)),0)$.

In numerical implementations, terms of the form $df^Tw$, with $f$ a vector field and $w$ a vector, require a long computational time since every partial derivative of $f$ has to be computed. In our context however, using the fact that the vector field $f(q,p)$ is Hamiltonian, the computations can be simplified in a substantial way. Indeed, using the commutation of partial derivatives, we get
$$
df_{(q,p)}^TZ=\begin{pmatrix}
 \nabla_q( f(q,p),Z) \\
\nabla_p( f(q,p),Z) 
\end{pmatrix}
=
\begin{pmatrix}
\nabla_q(\nabla_ph^Tz-\nabla_qh^T\alpha)\\
\nabla_p(\nabla_ph^Tz-\nabla_qh^T\alpha)
\end{pmatrix}
=
\begin{pmatrix}
\partial_p(\nabla_qh). z-\partial_q(\nabla_qh). \alpha)\\
\partial_p(\nabla_ph). z-\partial_q(\nabla_ph). \alpha)
\end{pmatrix} .
$$
Replacing $h$ with its expression, we get
$$
df_{(q,p)}^TZ
=
\begin{pmatrix}
\nabla_{q}(p^TK_qz)-\frac{1}{2}\partial_q(\nabla_q(p^TK_qp)).\alpha\\
K_qz-\partial_q(K_qp).\alpha
\end{pmatrix}.
$$
Therefore, instead of computing $d(\nabla_q(p^TK_qp))^T\alpha$, which requires the computation of all partial derivatives of $\nabla_q(p^TK_q)$, it is required to compute only one of them, namely the one with respect to $\alpha$. Let us sum up the result in the following proposition.

\begin{proposition}\label{prop8}
We have $ \nabla\hat{J}_1(p_0)=K_{q_0}p_0+\alpha(1)$, where $(z(\cdot),\alpha(\cdot))$ is the solution of
\begin{equation*}
\begin{split}
\dot{z}(t)&=\nabla_{q}(p(1-t)^TK_{q(1-t)}z(t))-\frac{1}{2}\partial_q(\nabla_q(p(1-t)^TK_{q(1-t)}p(1-t))).\alpha(t), \\
\dot{\alpha}(t)&=K_{q(1-t)}z(t)-\partial_q(K_{q(1-t)}p(1-t)).\alpha(t)
\end{split}
\end{equation*}
with $(z(0),\alpha(0))=(\nabla g(q(1)),0)$, and $(q(t),p(t))$ satisfies the geodesic equations $\dot{q}(t)=K_{q(t)}p(t)$ and $\dot{p}(t) = -\frac{1}{2}\nabla_q(p(t)^TK_{q(t)}p(t))$ for almost every $t\in[0,1]$, with $q(0)=q_0$ and $p(0)=p_0$.
\end{proposition}

A gradient descent algorithm can then be used in order to minimize $\hat{J}_1$ and thus $J_1$.

\subsection{Problems with constraints}\label{matchconst}
In this section, we derive several different methods devoted to solve numerically constrained optimal control problems on shape spaces. To avoid using overly technical notation in their whole generality, we restrict ourselves to the finite-dimensional case. The methods can however be easily adapted to infinite-dimensional shape spaces. We use the notation, the framework and the assumptions of Section \ref{modshsp}.

Let $X=\R^n$ and let $M$ be an open subset of $X$. For every $q\in M$, we identify $K_q$ with a $n\times n$ symmetric positive semi-definite real matrix. Throughout the section, we focus on kinetic constraints and we assume that we are in the conditions of Proposition \ref{equiv}, so that these constraints take the form $C_{q}K_{q}u=0$. Note that, according to Proposition \ref{proposition1}, in this case Problems \ref{diffprob1} and \ref{diffprob2} are equivalent.
Hence in this section we focus on Problem \ref{diffprob2}, and thanks to the identifications above the functional $J_2$ defined by \eqref{defJ2} can be written as
$$
J_2(u)=\frac{1}{2}\int_0^1 u(t)^TK_{q(t)}u(t) \, dt+g(q(1)).
$$
Note (and recall) that pure state constraints, of the form $C(q)=0$, are treated as well since, as already mentioned, they are equivalent to the kinetic constraints $dC_{q}. K_{q}u=0$.

\paragraph{The augmented Lagrangian method.}
This method consists of minimizing iteratively unconstrained functionals in which the constraints have been penalized. Although pure state constraints are equivalent to kinetic constraints, in this approach they can also be treated directly. The method goes as follows.
In the optimal control problem under consideration, we denote by $\lambda:[0,1]\rightarrow \R^{k}$ the Lagrange multiplier associated with the kinetic constraints $C_qK_qu=0$ (its existence is ensured by Theorem \ref{PMP}).
We define the augmented cost function
$$J_A(u,\lambda_1,\lambda_2,\mu)=\int_0^1 L_A(q(t),u(t),\lambda(t),\mu)\, dt+g(q(1)), $$
where $L_A$, called augmented Lagrangian, is defined by
$$
L_A(q,u,\lambda,\mu)=L(q,u)-\lambda^TC_qK_qu+\frac{1}{2\mu}\vert C_qK_qu\vert^2,
$$
with, here, $L(q,u)=\frac{1}{2}u^TK_qu$.
Let $q_0\in M$ fixed. Choose an initial control $u_0$ (for example, $u_0=0$), an initial function $\lambda_{0}$ (for example, $\lambda_{0}=0$), and an initial constant $\mu_0>0$. At step $\ell$, assume that we have obtained a control $u_\ell$ generating the curve $q_\ell$, a function $\lambda_{\ell}:[0,1]\rightarrow \R^{k}$, and a constant $\mu_\ell>0$. The iteration $\ell\rightarrow\ell+1$ is defined as follows.
First, minimizing the unconstrained functional $u\mapsto J_A(u,\lambda_{\ell},\mu_\ell)$ over $L^2(0,1;\R^n)$ yields a new control $u_{\ell+1}$, generating the curve $q_{\ell+1}$ (see further in this section for an appropriate minimization method). Second, $\lambda$ is updated according to
$$
\lambda_{\ell+1}=\lambda_{\ell}-\frac{1}{\mu_\ell}C_{q_{\ell+1}}K_{q_{\ell+1}}u_{\ell+1}.
$$
Finally, we choose $\mu_{\ell+1}\in (0,\mu_\ell]$ (many variants are possible in order to update this penalization parameter, as is well-known in numerical optimization).

Under some appropriate assumptions, as long as $\mu_\ell$ is smaller than some constant $\beta>0$, $u_\ell$ converges to a control $u^*$ which is a constrained extremum of $J_2$. Note that it is not required to assume that $\mu_\ell$ converge to $0$. More precisely we infer from \cite[Chapter 3]{IKBOOK} the following convergence result.

\begin{proposition}[Convergence of the augmented Lagrangian method]
Assume that all involved mappings are least of class $\mathcal{C}^2$ and that $C_qK_q$ is surjective for every $q\in M$.
Let $u^*$ be an optimal solution of Problem \ref{diffprob2} and let $q^*$ be its associated curve. Let $\lambda^*$ be the Lagrange multiplier (given by Theorem \ref{PMP}) associated with the constraints. We assume that there exist $c>0$ and $\mu>0$ such that 
\begin{equation}\label{auglag}
(\partial^2_{u}J_A)_{(u^*,\lambda^*,\mu)}.(\delta u,\delta u)\geq c\Vert \delta u\Vert_{L^2(0,1;\R^n)}^2,
\end{equation}
for every $\delta u\in L^2(0,1;\R^n)$.
Then there exists a neighborhood of $u^*$ in $L^2(0,1;\R^n)$ such that, for every initial control $u_0$ in this neighborhood, the sequence $(u_\ell)_{\ell\in\N}$ built according to the above algorithm converges to $u^*$, and the sequence $(\lambda_\ell)_{\ell\in\N}$ converges to $\lambda^*$, as $\ell$ tends to $+\infty$.
\end{proposition}

\begin{remark}
Assumption \eqref{auglag} may be hard to check for shape spaces. As is well-known in optimal control theory, this coercivity assumption of the bilinear form $(\partial^2_{u}J_A)_{(u^*,\lambda^*,\mu)}$ is actually equivalent to the nonexistence of \textit{conjugate points} of the optimal curve $q^*$ on $[0,1]$ (see \cite{BCT_COCV2007,BFT} for this theory and algorithms of computation). In practice, computing conjugate points is a priori easy since it just consists of testing the vanishing of some determinants; however in our context the dimension $n$ is expected to be large and then the computation  may become difficult numerically. 
\end{remark}

\begin{remark}
Pure state constraints $C(q)=0$ can either be treated in the above context by replacing $C_qK_q$ with $dC_q. K_q$, or can as well be treated directly by replacing $C_qK_q$ with $C(q)$ in the algorithm above.
\end{remark}

Any of the methods described in Section \ref{unconsmin} can be used in order to minimize the functional $J_A$ with respect to $u$. For completeness let us compute the gradient in $u$ of $J_A$ at the point $(u,\lambda,\mu)$. 

\begin{lemma}\label{lemma_computation_nabla_JA}
There holds
$$
\nabla_u J_A(u,\lambda,\mu) = u+C_q^T\lambda+\frac{1}{\mu}C_q^TC_qK_qu-p,
$$
where $p(\cdot)$ is the solution of
\begin{equation*}
\begin{split}
\dot{p}(t) &= -\partial_q(p(t)^TK_{q(t)}u(t))+\partial_qL_A(q(t),u(t),\lambda(t),\mu) \\
&= \partial_q\left(\left(\frac{u(t)}{2}-p(t)\right)^TK_{q(t)}u(t)\right)+\lambda(t)^T\partial_qC_{q(t)}u(t)+\frac{1}{2\mu}\partial_q\left(u(t)^TK_{q(t)}C_{q(t)}^TC_{q(t)}K_{q(t)}u(t)\right)
\end{split}
\end{equation*}
for almost every $t\in[0,1]$ and $p(1)+dg_{q(1)}=0$.
\end{lemma}

\begin{remark}
For pure state constraints $C(q)=0$, there simply holds $\nabla_u J_A=u-p$ and the differential equation in $p$ is
$$
\dot{p}(t)=\partial_q\left(\left(\frac{u(t)}{2}-p(t)\right)^TK_{q(t)}u(t)\right)+\lambda(t)^TdC_{q(t)}+\frac{1}{\mu}C(q(t))^TdC_{q(t)}.
$$
\end{remark}

\begin{proof}[Proof of Lemma \ref{lemma_computation_nabla_JA}]
We use Proposition \ref{dcost} with $L(q,u,t)=L_A(q,u,\lambda(t),\mu)$, with $\lambda$ and $\mu$ fixed (it is indeed easy to check that this proposition still holds true when the Lagrangian also depends smoothly on $t$). The differential of $J_A$ with respect to $u$ is then given by
\begin{equation}\label{dcost1}
d (J_A)_u . \delta u=\int_0^1\left(\partial_uL_A(q(t),u(t),\lambda(t),\mu)-p(t)^TK_{q(t)}\right)\delta u(t) \, dt,
\end{equation}
where $p(\cdot)$ is the solution of $\dot{p}(t)=-\partial_q(p(t)^TK_{q(t)}u(t))+\partial_qL_A(q(t),u(t),\lambda(t),\mu)$ for almost every $t\in[0,1]$ and $p(1)+dg_{q(1)}=0$. To get the result, it then suffices to identify the differential $d (J_A)_u$ with the gradient $\nabla J_A(u)$ with respect to the inner product on $L^2(0,1;R^n)$ given by $(u_1,u_2)_{L^2(0,1;R^n)} = \int_0^1 u_1(t)^TK_{q(t)}u_2(t)\, dt$.
\end{proof}

The advantage of the augmented Lagrangian method is that, at every step, each gradient is ``easy'' to compute (at least as easy as in the unconstrained case). The problem is that, as in any penalization method, a lot of iterations are in general required in order to get a good approximation of the optimal solution, satisfying approximately the constraints with enough accuracy.

The next method we propose tackles the constraints without penalization.

\paragraph{Constrained minimization through shooting.}
We adapt the usual shooting method used in optimal control (see, e.g., \cite{TBOOK}) to our context. For $(q,p)\in M\times X^*=M\times\R^n$, we define $\lambda_{q,p}$ as in Theorem \ref{geodeq} by
$$
\lambda_{q,p}=(C_qK_qC_q)^{-1}C_qK_qp.
$$ 
We also denote $\pi_qp=p-C_q^T\lambda_{q,p}$. In particular, $v_{q,p}=K_q\pi_qp$. 

\begin{remark}
A quick computation shows that $\pi_qp$ is the orthogonal projection of $p$ onto $\ker (C_q)$ for the inner product induced by $K_q$. 
\end{remark}

According to Theorem \ref{geodeq}, Corollary \ref{geodeq3} and Remark \ref{rem_propequiv} (see Section \ref{geodeqshsp}), the minimizers of $J_2$ have to be sought among the geodesics $(q(\cdot),p(\cdot))$ solutions of \eqref{geodeq1}, and moreover $p(0)$ is a minimizer of the functional
$$
\hat{J}_2(p_0)=\frac{1}{2}\Vert v_{q_0,p_0}\Vert_V^2+g(q(1))=\frac{1}{2}p_0^T\pi_{q_0}K_{q_0}\pi_{q_0}p_0+g(q(1)).
$$
The geodesic equations \eqref{geodeq1} now take the form
\begin{equation*}
\begin{split}
\dot{q}(t)&=\partial_ph(q(t),p(t)) = K_{q(t)}\pi_{q(t)}p(t), \\
\dot{p}(t)&= -\partial_qh(q(t),p(t)) = -\frac{1}{2}p(t)^T\pi_{q(t)}^T(\nabla_qK_{q(t)})\pi_{q(t)}p(t) + \lambda_{q(t),p(t)}^T(\nabla_qC_{q(t)})K_{q(t)}\pi_{q(t)}p(t),
\end{split}
\end{equation*}
where $h$ is the reduced Hamiltonian (see Remark \ref{rem11}).
It follows from Proposition \ref{prop8} that $\nabla \hat{J}_2(p_0)=\pi_{q_0}^TK_{q_0}\pi_{q_0}p_0+\alpha(1)$, where $Z(\cdot)=(z(\cdot),\alpha(\cdot))$ is the solution of
\begin{equation*}
\begin{split}
\dot{z}(t) &= \partial_p(\nabla_qh(q(1-t),p(1-t))). z(t)-\partial_q(\nabla_qh(q(1-t),p(1-t))). \alpha(t) \\
\dot{\alpha}(t) &= \partial_p(\nabla_ph(q(1-t),p(1-t))). z(t)-\partial_q(\nabla_ph(q(1-t),p(1-t))). \alpha(t)
\end{split}
\end{equation*}
with $(z(0),\alpha(0))=(\nabla g(q(1)),0)$. Replacing $\nabla_qh$ and $\nabla_ph$ by their expression, we get 
\begin{equation*}
\begin{split}
\dot{z}(t) &= p(1-t)^T\pi_{q(1-t)}^T(\nabla_qK_{q(1-t)})\pi_{q(1-t)}z(t) -\lambda_{q(1-t),p(1-t)}^T(\nabla_qC_{q(1-t)})K_{q(1-t)}\pi_{q(1-t)} z(t) \\
& \quad
-\partial_q\Big(\frac{1}{2}p(1-t)^T\pi_{q(1-t)}^T(\nabla_qK_{q(1-t)})\pi_{q(1-t)}p(1-t) \\
&\qquad\qquad
 -\lambda_{q(1-t),p(1-t)}^T(\nabla_qC_{q(1-t)})K_{q(1-t)}\pi_{q(1-t)}p(1-t)\Big). \alpha(t) \\
\dot{\alpha}(t) &=K_{q(1-t)}\pi_{q(1-t)}z(t)-\partial_q(K_{q(1-t)}\pi_{q(1-t)}p(1-t)). \alpha(t))
\end{split}
\end{equation*}
In practice, the derivatives appearing in these equations can be efficiently approximated using finite differences.

This algorithm of constrained minimization through shooting has several advantages compared with the previous augmented Lagrangian method. The first is that, thanks to the geodesic reduction, the functional $\hat{J}_2$ is defined on a finite-dimensional space (at least whenever the shape space itself is finite dimensional) and hence $\nabla\hat{J}_2(p_0)$ is computed on a finite-dimensional space, whereas in the augmented Lagrangian method $\nabla_u J_A$ was computed on the infinite-dimensional space $L^2(0,1;\R^n)$.

A second advantage is that, since we are dealing with constrained geodesics, all resulting curves satisfy the constraints with a good numerical accuracy, whereas in the augmented Lagrangian method a large number of iterations was necessary for the constraints to be satisfied with an acceptable numerical accuracy. 

This substantial gain is however counterbalanced by the computation of $\lambda_{q(t),p(t)}$, which requires solving of a linear equation at every time $t\in[0,1]$ along the curve (indeed, recall that $\lambda_{q,p}=(C_qK_qC_q^*)^{-1}C_qK_qp$). The difficulty here is not just that this step is time-consuming, but rather the fact that the linear system may be ill-conditioned, which indicates that this step may require a more careful treatment.
One possible way to overcome this difficulty is to solve this system with methods inspired from quasi-Newton algorithms. This requires however a particular study that is beyond the scope of the present article (see \cite{A1} for results and algorithms).

\section{Numerical examples}\label{sec7}
\subsection{Matching with constant total volume}
In this first example we consider a very simple constraint, namely, a constant total volume. Consider $S=S^{d-1}$, the unit sphere in $\R^d$, and let $M=\mathrm{Emb}^1(S,\R^d)$ be the shape space, acted upon with order $1$ by $\mathrm{Diff}(\R^d)$.
Consider as in \cite{TY2} the RKHS $V$ of smooth vector fields given by the Gaussian kernel $K$ with positive scale $\sigma$ defined by $K(x,y)=e^{-\frac{\vert x-y\vert^2}{\sigma^2}}I_d$.
An embedding $q\in M$ of the sphere is the boundary of an open subset $U(q)$ with total volume given by $\mathrm{Vol}(U(q))=\int_S q^*\omega$, where $\omega$ is a $(d-1)$-form such that $d\omega=dx_1\wedge\dots\wedge dx_d$.
Let $q_0$ be an initial point and let $q_1$ be a target such that $\mathrm{Vol}(U(q_0))=\mathrm{Vol}(U(q_1))$. 
We impose as a constraint to the deformation $q(\cdot)$ to be of constant total volume, that is, $\mathrm{Vol}(U(q(t)))=\mathrm{Vol}(U(q_0))$. The data attachment function is defined by $g(q)=d(q,q_1)^2$, with $d$ a distance between submanifolds (see \cite{GV} for examples of such distances).

For the numerical implementation, we take $d=2$ (thus $S=S^1$) and $M$ is a space of curves, which is discretized as landmarks $q=(x_1,\dots,x_n)\in\mathrm{Lmk}_2(n)$. The volume of a curve is approximately equal to the volume of the polygon $P(q)$ with vertices $x_i$, given by $\mathrm{Vol}(P(q))=\frac{1}{2}(x_1y_2-y_2x_1+\dots+x_ny_1-y_nx_1)$.

If one does not take into account a constant volume constraint, a minimizing flow matching a circle on a translated circle usually tends to shrink it along the way (see Figure \ref{fig:compa_a}). If the volume is required to remain constant then the circle looks more like it were translated towards the target, though the diffeomorphism itself does not look like a translation (see Figure \ref{fig:compa_b}).

\begin{figure}[h]
\begin{center}
\subfigure[Matching trajectories without constraints. \label{fig:compa_a}]
{\includegraphics[width=7.43cm]{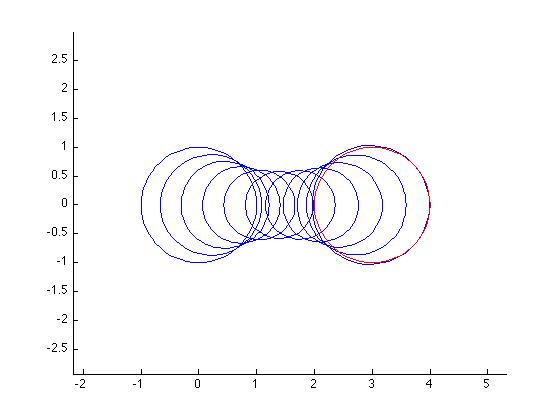}}
%\qquad\qquad\qquad
\subfigure[Matching trajectories with constant volume.]
{\label{fig:compa_b}\includegraphics[width=7.43cm]{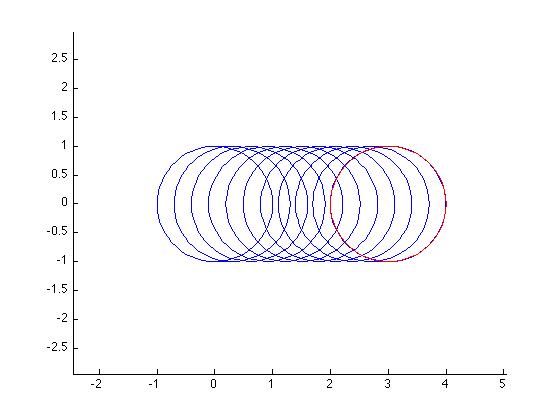}}
\caption{Matching trajectories.\label{fig:compa}}
\end{center}
\end{figure}

%\begin{figure}[h]
%\begin{center}
%\includegraphics[width=7.43cm]{sanscont.jpg}
%\includegraphics[width=7.43cm]{avcont.jpg}
%\caption{Matching trajectories without constraints (on the left) and with constant volume (on the right).}
%\label{fig:compa}
%\end{center}
%\end{figure}

The implementation of the shooting method developed in Section \ref{matchconst} leads to the diffeomorphism represented on Figure \ref{fig:CV}.
 
\begin{figure}[h]
\begin{center}
\subfigure[Initial condition (in blue) and target (in red). \label{fig:CV_a}]
{\includegraphics[width=7.43cm]{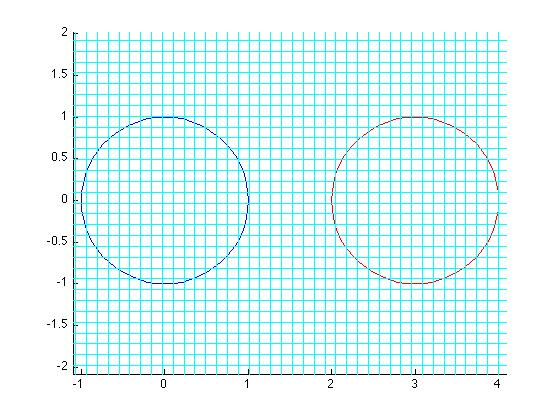}}
\subfigure[Matching. \label{fig:CV_b}]
{\includegraphics[width=7.43cm]{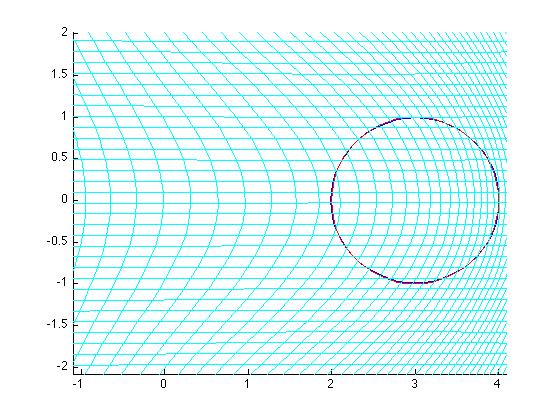}}
\caption{Constant volume experiment.}\label{fig:CV}
\end{center}
\end{figure}

\subsection{Multishape matching}
We consider the multishape problem described in Section \ref{mlmk}. We define the shape spaces $M_1, \ldots, M_k$, by $M_j = \mathrm{Emb}(S_j, \R^d)$ for every $j\in\{1,\ldots,k-1\}$ and $M_{k} = M_1 \times \cdots \times M_{k-1}$ (background space). For every $j\in \{1, \dots, k\}$ we consider a reproducing kernel $K_i$ and a reduced operator $K_{q,j}$ for every $q \in M_j$. 

In the following numerical simulations, each $q_j$ is a curve in $\R^2$, so that $S_1 = \cdots = S_k = S^1$, the unit circle. The function $g$ appearing in the functional \eqref{deffunctionalmultishapes} is defined by
$$
g(q_1, \dots, q_k) = \sum_{j=1}^{k-1} \left( d(q_j, q^{(j)})^2 + d(q_{k}^j, q^{(j)})^2 \right),
$$
where $q_{k} = (q_{k}^1, \ldots, q_{k}^{k-1})$ and $q^{(1)}, \ldots, q^{(k-1)}$ are given target curves. The distance $d$ is a distance between curves (see \cite{GV} for examples of such distances). We consider two types of compatibility constraints between homologous curves $q_j$ and $q_{k}^j$: either the identity (or stitched) constraint $q_j = q_{k}^j$, or the identity up to reparametrization (or sliding) constraint $q_{k}^j =  q_j\circ f$ for some (time-dependent) diffeomorphism $f$ of $S^1$.
Note that, since the curves have the same initial condition, the latter constraint is equivalent to imposing that $\dot q_{k}^j  - \dot q_j\circ f $ is tangent to $q_{k}^j$, which can also be written as
$$(v_{k}(t,  q_{k}^j) - v_{j}(t, q_{k}^j))\cdot \nu_{k}^j = 0,$$
where $v_j = K_j \xi^*_{q_j} u_j$ and $\nu_{k}^j$ is normal to $q_{k}^j$.
In the numerical implementation, the curves are discretized into polygonal lines, and the discretization of the control system \eqref{contsysmultishape} and of the minimization functional \eqref{deffunctionalmultishapes} is done by reduction to landmark space, as described in section \ref{approx}. The discretization of the constraint in the identity case is straightforward. For the identity up to reparametrization (or sliding) constraint, the discretization is slightly more complicated and can be done in two ways. A first way is to add a new state variable $\nu_{k}^j$ which evolves while remaining normal to $q_{k}^j$, according to
$$ \dot \nu_{k}^j = -dv_{k}^j(q_{k}^j)^T \nu_{k}^j,$$
which can be written as a function of the control $u_{k}$ and of the derivatives of $K_{k}$ (this is an example of lifted state space, as discussed in Section \ref{mlmk}).
A second way, which is computationally simpler and that we use in our experiments, avoids introducing a new state variable and uses finite-difference approximations. For every $j=1, \ldots, k-1$, and every line segment $\ell = [z_\ell^-, z_\ell^+]$ in $q_{k}^{j}$ (represented as a polygonal line), we simply use the constraint
$ \nu_\ell \cdot (v_j(\ell) - v_{k}(\ell)) = 0 $,
where $\nu_\ell$ is the unit vector perpendicular of $z_\ell^+ - z_\ell^-$ and 
$v_j(\ell) = \frac{1}{2}(v_j(z_\ell^-) + v_j(z_\ell^+))$.
Note that the vertices $z_\ell^-$ and $z_\ell^+$ are already part of the state variables that are obtained after discretizing the background boundaries $q_{k}$.

\begin{figure}[h]
\begin{center}
\includegraphics[width=9cm]{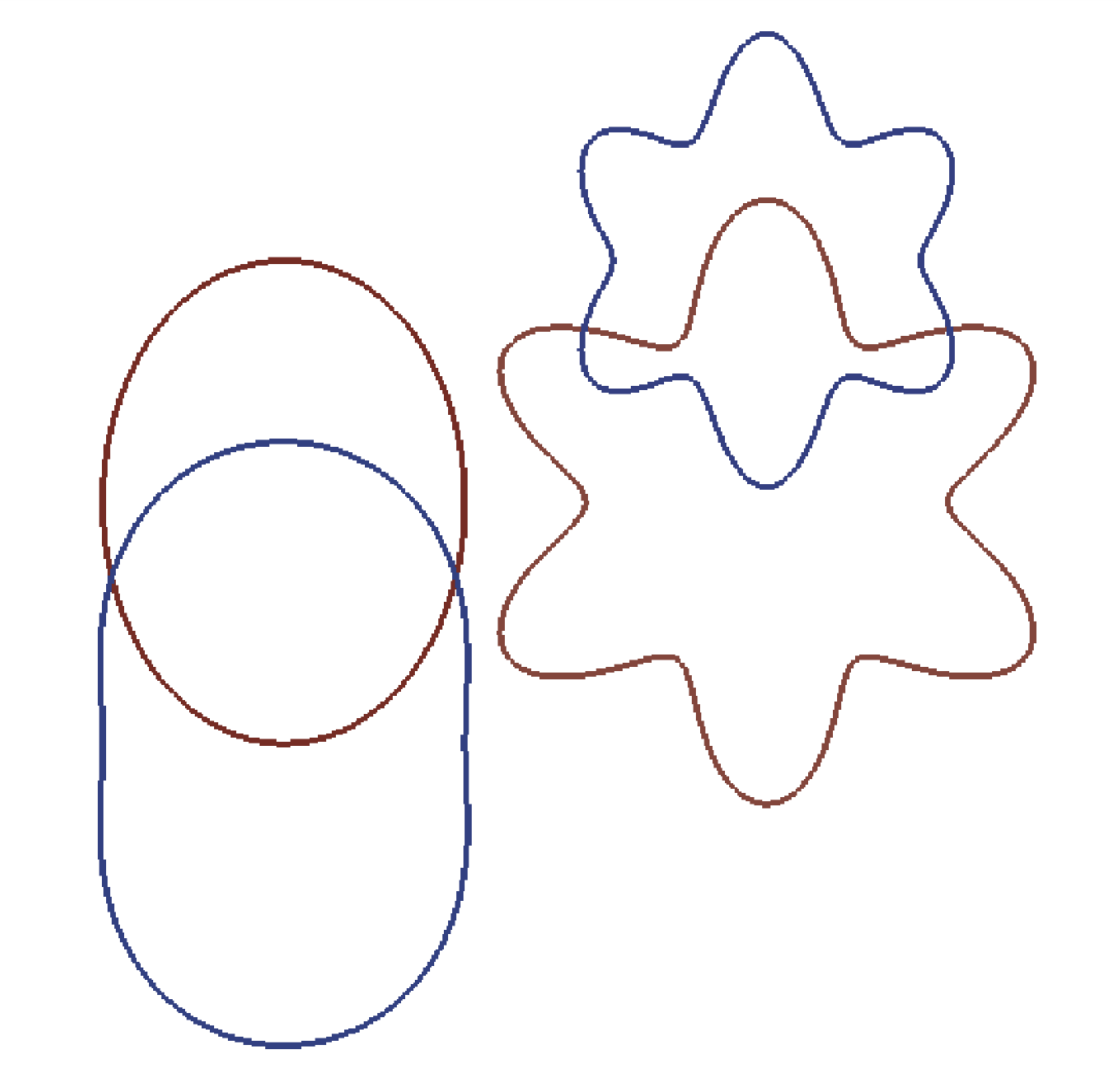}
\caption{Multishape Experiment: Initial (blue) and target (red) sets of curves.}
\label{fig:multishape.1}
\end{center} 
\end{figure}

With these choices, the discretized functional and its associated gradient for the augmented Lagrangian method are obtained with a rather straightforward -- albeit lengthy -- computation. In Figure \ref{fig:multishape.1}, we provide an example comparing the two constraints. In this example, we take $k=2$ and use the same radial kernel $K_1 = K_2$ for the two shapes, letting $K_1(x,y) = \gamma(|x-y|/\sigma_1)$, with
\[
\gamma(t) = (1 + t + 2t^2/5 + t^3/15)e^{-t}.
\]
The background kernel is  $K_3(x,y) = \gamma(|x-y|/\sigma_3)$, with $\sigma_1 = 1$ and $\sigma_3=0.1$.  The desired transformation, as depicted in Figure \ref{fig:multishape.1}, moves a curve with elliptical shape upwards, and a flower-shaped curve downwards, each curve being, in addition, subject to a small deformation. The compared curves have a diameter of order 1.

The solutions obtained using  the stitched and sliding constraints are provided in Figures \ref{fig:multishape.2_a} and \ref{fig:multishape.2_b}, in which we have also drawn a deformed grid representing the diffeomorphisms induced by the vector fields $v_1$, $v_2$ and $v_3$ in their relevant regions. The consequence of the difference between the kernel widths inside and outside the curves on the regularity of the deformation is obvious in both experiments. One can also note differences in the deformation inside between the stitched and sliding cases, the second case being more regular thanks to the relaxed continuity constraints at the boundaries. Finally, we mention the fact that the numerical method that we illustrate here can be easily generalized to triangulated  surfaces instead of polygonal lines.

\begin{figure}[H]
\begin{center}
\subfigure[Stitched constraints. \label{fig:multishape.2_a}]
{\includegraphics[width=7.cm]{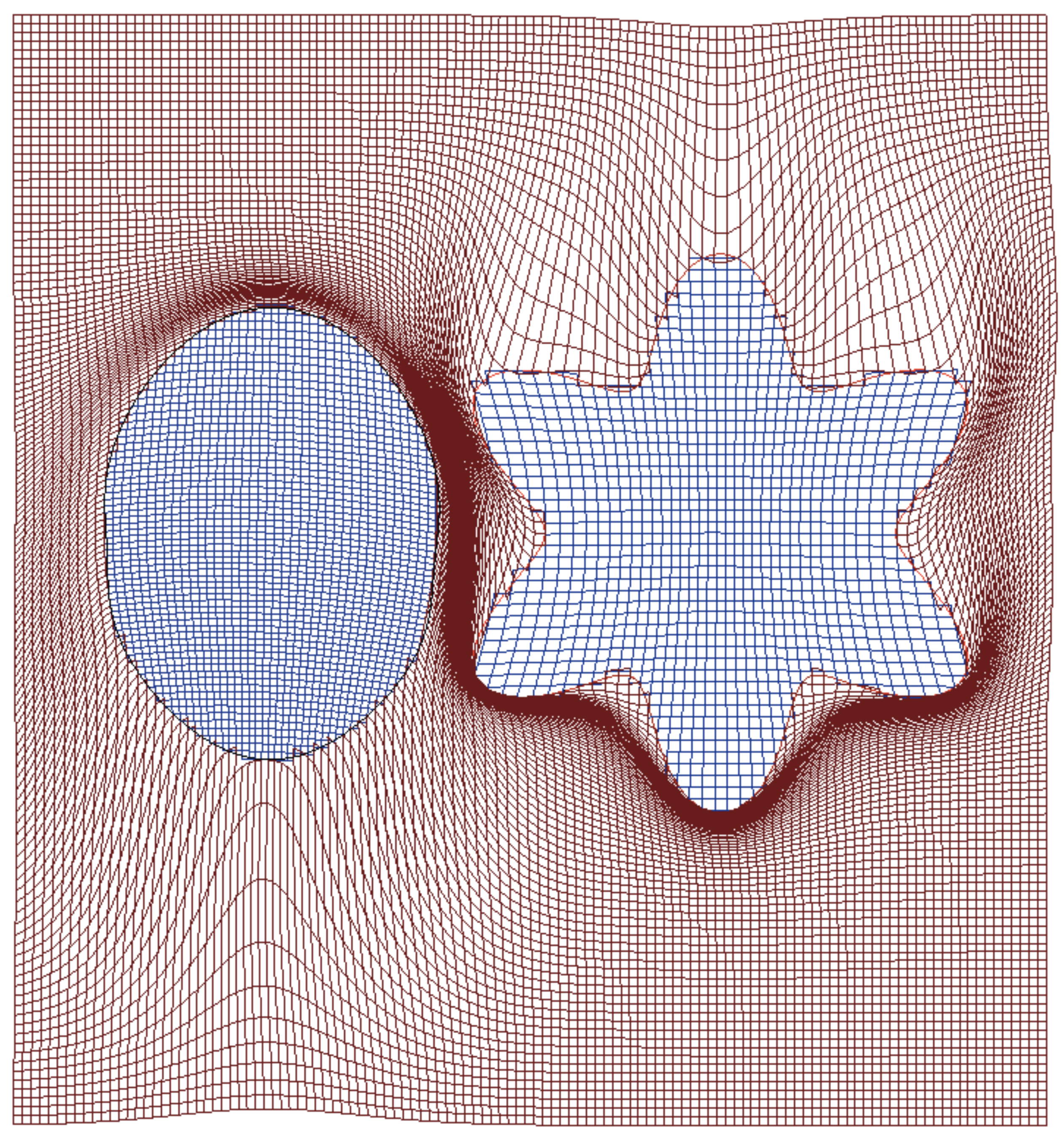}}
\subfigure[Sliding constraints. \label{fig:multishape.2_b}]
{\includegraphics[width=7.cm]{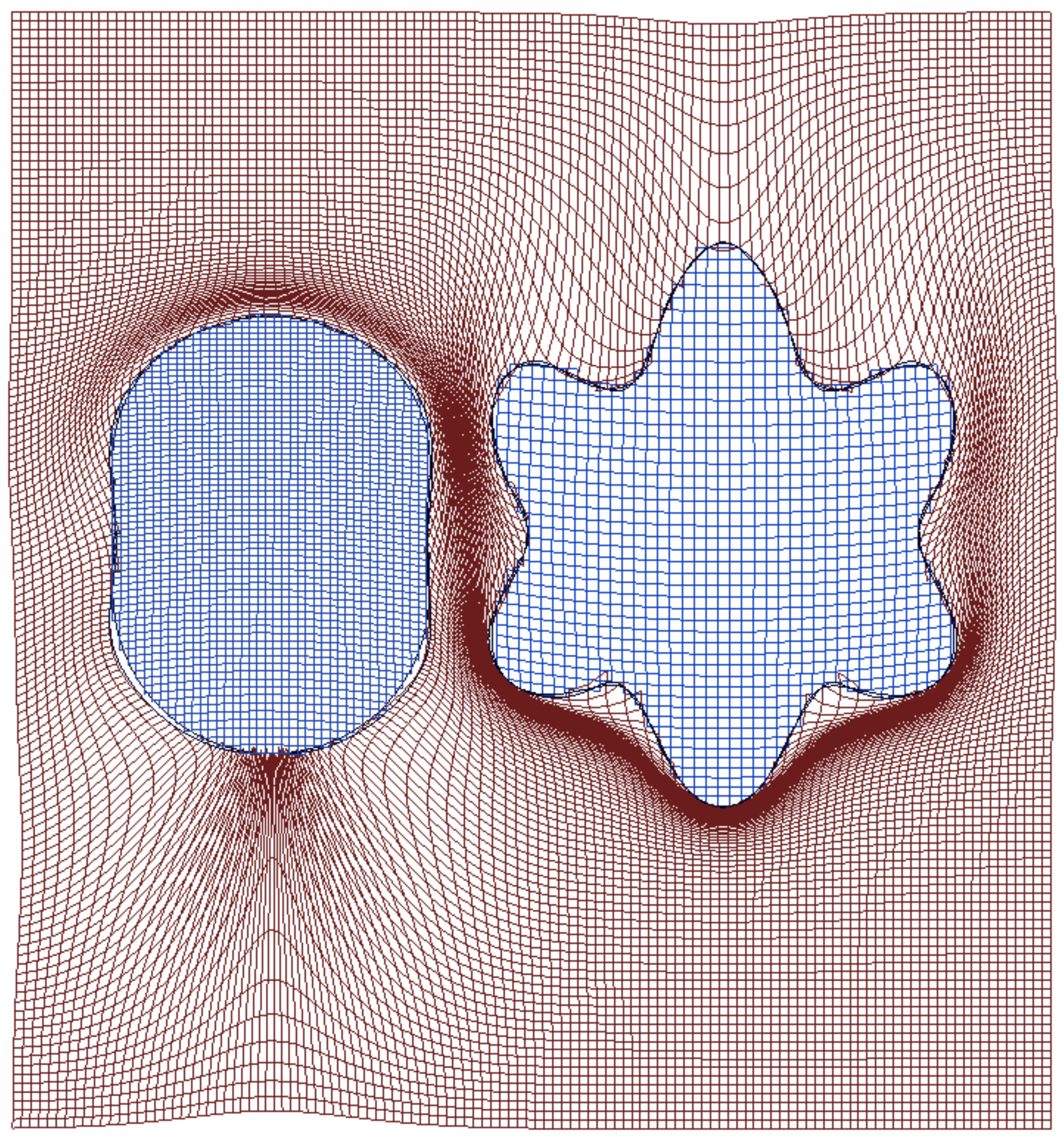}}
\caption{Multishape Experiment.}\label{fig:multishape.2}
\end{center}
\end{figure}

\section{Conclusion and open problems} 
The purpose of this paper was to develop a very general framework for the analysis of shape deformations, along with practical methods to find an optimal deformation in that framework. The point of view of control theory gives powerful tools to attain this goal. This allows in particular for the treatment of constrained deformation, which had not been done before.

Now that a concrete setting has been fixed, many new developments can be expected. First of all, the minimization algorithms in the case of constrained shapes are quite slow for a very high number of constraints. Moreover, we did not study any of the geometric aspects of shape deformation spaces. For example, we only briefly mentioned the infinite-dimensional sub-Riemannian structure that the RKHS induces on both groups of diffeomorphism and shape spaces. Sub-Riemannian geometry in infinite dimension and codimension is still a very open subject with very few results. The sub-Riemannian geometry in this paper is particularly difficult and interesting to study because the horizontal spaces may not be closed in the ambient space. A key difference with finite dimension is that some geodesics might exist that are neither normal nor abnormal.

More general control problems can be designed. One can, for example, use a second-order approach, with a control system taking the form  
$$
\dot{q}=\partial_{p}H(q,p),\quad \dot{p}=-\partial_{q}H(q,p) + f(q,u)
$$
in which the original state is lifted to the cotangent space ($q\to (q,p)$), and the new control is $u$. These models have been introduced for shapes in \cite{TV}, with $f(q,u)=u$, providing a way to interpolate smoothly between multiple shapes. We are currently exploring applications of this approach to model muscle-like motions, with external forces constrained to being collinear to the fibers. 
%Another topic left relatively untouched is that of lifted shapes. Here is an example of application for such a space, inspired by \cite{TV}. In $\mathrm{Lmk}_d(n)$ and for an RKHS $V$ with kernel $K$, it is known that the geodesic equations for $(q,p)=(x_1,\dots,x_n,p_1,\dots,p_n)$ are given by
%$$
%\dot{x}_i=\partial_{p_i}H(q,p),\quad \dot{p}_i=-\partial_{x_i}H(q,p).
%$$
%The term $-\partial_{x_i}H(q,p)$ can be interpreted as the internal forces of the system. Now assume that we are modeling a muscle. Then the contraction of the muscle adds to these internal forces a new force  $F_i=f_iL_i,\ f_I\in\R,\ L_i\in\R^d,$ in the direction of the fiber $L_i$ of the muscle at $x_i$, transforming the second equation to $-\partial_{x_i}H(q,p)+f_iL_i$. However, since the muscle is deformed by the flow of diffeomorphisms induced by the vector field 
%$$
%v(x)=(K_V\xi_q^*p)(x)=\sum_{i=1}^nK(x_i,x)p_i,
%$$
%it follows that the fibers are moved by its derivative
%$$
%\dot{L}_i=dv_{x_i}\dot L_i=\sum_{i=1}^n(\partial_2K(x_i,x)\dot L_i)p_i.
%$$
%Here, the equivariant submersion is given by $P(q,L)=q$, leading to a new control system that models the deformation of muscle when it contracts, which is written as
%$$
%\dot{x}_i=\partial_{p_i}H(q,p),\quad \dot{p}_i=-\partial_{x_i}H(q,p)+f_iL_i,\quad
%\dot{L}_i=\sum_{i=1}^n(\partial_2K(x_i,x)\dot L_i)p_i .
%$$
%}

Another glaring issue comes from the assumption of the surjectivity of the constraints in Theorem \ref{PMP}. Indeed, in most practical cases, such as multishapes in Section \ref{mlmk}, this assumption fails. In finite dimension, problems can occur when the rank of $C_q$ changes with $q$. They are usually solved by taking higher-order derivatives of $C_q$ on the sets on which it is not maximal. This does not seem easily possible with Banach spaces. Another problem comes from the incompatibility of topologies between the Hilbert space $V$ and the Banach space $Y$ in which the constraints are valued: $C_q$ may not have a closed range, in which case we could find "missing" Lagrange multipliers. If $C_q(V)$ were constant, this could be solved simply by restricting $Y$ to $C_q(V)$ equipped with the Hilbert topology induced by $C_q$ and $V$, but since it is not constant, this might be impossible. It would be very interesting, both for control theory in general and for shape deformation analysis in particular, to find a way to address this problem.

\end{document}